\documentclass[a4paper, 11pt]{amsart}

\usepackage{geometry}

\usepackage{graphicx}

\usepackage[british,UKenglish,USenglish,english,american]{babel}
\usepackage{newlfont}
\usepackage{amssymb}
\usepackage{amsmath}
\usepackage{amsfonts}
\usepackage{latexsym}
\usepackage{amsthm}
\usepackage{mathrsfs}
\usepackage{stmaryrd}
\usepackage{caption}
\usepackage{enumerate}

\usepackage{tikz-cd}

\usepackage{amsmath,amssymb,amscd,xypic}
\usepackage{amsfonts}
\usepackage[english]{babel}
\usepackage[leqno]{amsmath}
\usepackage{mathrsfs} 
\usepackage{enumerate}

\usepackage{layout}
\usepackage{fullpage}

\usepackage[backref=page]{hyperref}

\hypersetup{
 colorlinks,
 citecolor=green,
 linkcolor=blue,
 urlcolor=blue}

\theoremstyle{plain}                    
\newtheorem{teo}{Theorem}[section]     

\newtheorem{prop}[teo]{Proposition}

\newtheorem{fact}[teo]{Fact}
\newtheorem{cor}[teo]{Corollary}       

\newtheorem*{ThmA}{Theorem A}
\newtheorem*{ThmB}{Theorem B}

\newtheorem{lem}[teo]{Lemma}            
\theoremstyle{definition}               

\newtheorem*{convention}{Convention}
\newtheorem{defin}[teo]{Definition}

\theoremstyle{remark}
\newtheorem{rmk}[teo]{Remark}

\numberwithin{equation}{section}
\newenvironment{sis}{\left\{\begin{aligned}}{\end{aligned}\right.}

\newcommand{\bbL}{{\mathbb L}}

\newcommand{\bbC}{{\mathbb C}}
\newcommand{\bbQ}{{\mathbb Q}}
\newcommand{\bbZ}{{\mathbb Z}}
\newcommand{\bbP}{{\mathbb P}}

\newcommand{\cB}{{\mathcal B}}
\newcommand{\cC}{{\mathcal C}}

\renewcommand{\cL}{{\mathcal L}}

\newcommand{\cM}{{\mathcal M}}

\newcommand{\cO}{{\mathcal O}}

\newcommand{\cX}{{\mathcal X}}
\newcommand{\cY}{{\mathcal Y}}

\newcommand{\un}{\underline}
\newcommand{\ov}{\overline}
\newcommand{\wt}{\widetilde}

\renewcommand{\rm}{\mathrm}

\numberwithin{equation}{subsection}

\newcommand{\oo}{\mathcal{O}}
\newcommand{\mt}{\mathcal}

\DeclareMathOperator{\Aut}{Aut}
\DeclareMathOperator{\Spec}{Spec}

\DeclareMathOperator{\RPic}{RPic}
\DeclareMathOperator{\NS}{NS}
\DeclareMathOperator{\Hom}{Hom}

\DeclareMathOperator{\coker}{coker}
\renewcommand{\Im}{\text{Im}}

\DeclareMathOperator{\id}{id}
\DeclareMathOperator{\res}{res}
\DeclareMathOperator{\red}{red}
\DeclareMathOperator{\irr}{irr}
\DeclareMathOperator{\rk}{rk}
\DeclareMathOperator{\ta}{taut}
\DeclareMathOperator{\tor}{tor}
\DeclareMathOperator{\tf}{tf}
\renewcommand{\char}{\text{char}}
\DeclareMathOperator{\Gm}{\mathbb{G}_{m}}

\DeclareMathOperator{\Ga}{\mathbb{G}_{\mathrm a}}
\DeclareMathOperator{\PGL}{PGL}

\DeclareMathOperator{\Sch}{Sch}

\DeclareMathOperator{\AbGr}{AbGrps}

\DeclareMathOperator{\an}{an}

\DeclareMathOperator{\rel}{rel}
\DeclareMathOperator{\sm}{sm}

\newcommand{\Mg}{\mathcal M_{g,n}}
\newcommand{\Mgb}{\ov{\mathcal M}_{g,n}}
\newcommand{\Cg}{\mathcal C_{g,n}}
\newcommand{\Cgb}{\ov{\mathcal C}_{g,n}}

\DeclareMathOperator{\Pic}{Pic}
\DeclareMathOperator{\Cl}{Cl}
\DeclareMathOperator{\PIC}{\textbf{Pic}}
\DeclareMathOperator{\NSb}{\textbf{NS}}

\newcommand{\bl}{\bullet}
\newcommand{\lie}{\operatorname{Lie}}

\newcommand{\wpr}{\underline{W}}
\newcommand{\fwr}{\widehat{\underline{W}}}

\title{On the Picard group scheme of the moduli stack of stable pointed curves}

\author{Roberto Fringuelli}
\address{Roberto Fringuelli, Dipartimento di Matematica, Universit\`a di Roma La Sapienza, Piazzale Aldo Moro 5, I-00185 Roma Italy}
\email{r.fringuelli@uniroma1.it}

\author{Filippo Viviani}
\address{Filippo Viviani,
Dipartimento di Matematica e Fisica 
Universit\`a Roma Tre 
Largo San Leonardo Murialdo  
I-00146 Roma  Italy }
\email{viviani@mat.uniroma3.it}

\begin{document}



\begin{abstract}
The aim of the present paper is to study the (abstract) Picard group and the Picard group scheme  of the moduli stack of stable pointed curves over an arbitrary scheme. As a byproduct, we compute the Picard groups of the moduli stack of stable or smooth pointed curves over a field of  characteristic different from two. 
\end{abstract}

\maketitle


\section{Introduction.}

For any pair of integers $g,n\geq 0$ such that $3g-3+n>0$, let  $f:\Mgb\to \Spec \bbZ$ be the stack of stable $n$-pointed curves of genus $g$ and denote by  $f_S:\Mgb^S\to S$ its base change over a scheme $S$. 
The aim of this paper is to describe the (abstract) Picard group $\Pic(\Mgb^S)$ of $\Mgb^S$ and the Picard group (algebraic) space $\PIC_{\Mgb^S/S}$  (which turns out to be a group scheme, see Proposition \ref{P:propPic}) of $f_S:\Mgb^S\to S$, which represents the fppf sheafification  of the relative Picard group functor of $f_S$ (see Section \ref{Sec:PicSpa}).


In order to describe our main result, let us fix some notation. Denote by $\Lambda_{g,n}$ the abelian group  generated by 
$$\un{\lambda}, \: \un{\delta_{\irr}} \: \text{ and } \{\un{\delta_{a,A}}\: : 0\leq a \leq g, \emptyset \subseteq A\subseteq [n]:=\{1,\ldots, n\} \text{ with } (a,A)\neq (0,\emptyset), (g, [n])\} $$
subject to the following relations
\begin{equation}\label{E:rela}
\begin{sis}
& \un{\delta_{a,A}}=\un{\delta_{g-a,A^c}}, \\
& 10\un{\lambda}=\un{\delta_{\irr}}+2\sum_{\substack{A\subseteq  [n]}}\un{\delta_{1,A}} & \text{ if } g=2, \\
& 12\un{\lambda}=\un{\delta_{\irr}}& \text{ if } g=1, \\
& \un{\lambda}+\sum_{p\in A} \un{\delta_{0,A}}=0 \: \text{ for any } 1\leq p \leq n & \text{ if } g=1, \\
& \un{\lambda}=\un{\delta_{\irr}}=0& \text{ if } g=0,\\
& \sum_{\substack{z\in A\\x,y\not\in A}}\un{\delta_{0,A}}=0 \: \text{ for any pairwise distinct } x,y,z\in \{1,\ldots, n\}  & \text{ if } g=0,\\
& \sum_{\substack{p,q\in A \\r,s\not\in A}}\un{\delta_{0,A}}= \sum_{\substack{p,r\in A \\q,s\not\in A}}\un{\delta_{0,A}} \: \text{ for any pairwise distinct } p,q,r,s\in \{1,\ldots, n\}  & \text{ if } g=0.\\
\end{sis}
\end{equation}
Note that $\Lambda_{g,n}$ is a free abelian group of finite rank.  For any scheme $S$, we will denote by $\un{\Lambda_{g,n}}_S$ the constant commutative group scheme over $S$ associated to $\Lambda_{g,n}$.

In Section \ref{S:homPhi}, we show that there exists a group homomorphism (that we call the \emph{tautological homomorphism})
\begin{equation}\label{E:Pgnbar}
\ov P_{g,n}:\Lambda_{g,n}\to \Pic(\Mgb),
\end{equation}
that sends each generator of $\Lambda_{g,n}$ into the tautological line bundle on $\Mgb$ with the same name, using the standard convention that $\delta_{0,\{i\}}=\delta_{g,\{i\}^c}:=-\psi_i$.
For any scheme $\phi:S\to \Spec \bbZ$, we define 
\begin{equation}\label{E:PgnSbar}
\ov P_{g,n}(S):\Lambda_{g,n}\xrightarrow{\ov P_{g,n}} \Pic(\Mgb)\xrightarrow{\ov \phi^*} \Pic(\Mgb^S),
\end{equation}
\begin{equation}\label{E:PgnSbarrel}
\ov P_{g,n}^{\rel}(S):\Lambda_{g,n}\xrightarrow{\ov P_{g,n}} \Pic(\Mgb)\xrightarrow{\ov \phi^*} \Pic(\Mgb^S)\twoheadrightarrow \frac{\Pic(\Mgb^S)}{f_S^*(\Pic S)},
\end{equation}
$\ov \phi:\Mgb^S\to \Mgb$ is the base change morphism induced by $\phi$.
 By construction, the  homomorphisms $\{\ov P^{\rel}_{g,n}(S)\}_S$ define a natural transformation from the constant presheaf $\Lambda_{g,n}$ onto the relative Picard functor of $f:\Mgb\to \Spec \bbZ$;
 hence, by passing to the associated sheaves (for the fppf topology), we get  a homomorphism of commutative group schemes,  compatible with base change,  that we call the \emph{tautological $S$-morphism}:
\begin{equation}\label{E:Phi}
\Phi^S_{g,n}:\un{\Lambda_{g,n}}_S\to \PIC_{\Mgb^S/S}.
\end{equation}

The first main result of this paper is the following 

\begin{ThmA}
Assume that $S$ is a scheme over $\Spec \bbZ[1/2]$ (or equivalently  that the residue field $k(s)$ of every $s\in S$ has characteristic different from $2$) or that $g\leq 5$.
\begin{enumerate}
\item  \label{T:mainThm1} The homomorphism $\Phi_{g,n}^S$ of \eqref{E:Phi} is an isomorphism.
\item \label{T:mainThm2} If  $S$ is connected, then there is an isomorphism (functorial in $S$)
$$
\Pic(S)\times \Lambda_{g,n} \xrightarrow[\cong]{f_S^*\times \ov P_{g,n}(S)}\Pic(\Mgb^S).
$$
\end{enumerate}
\end{ThmA}
The above Theorem A was proved for $\ov\cM_{1,1}$ by Fulton-Olsson \cite{FO} over an arbitrary scheme $S$ (not necessarily over $\Spec \bbZ[1/2]$). The proof in loc. cit. uses an explicit presentation of the stack  of elliptic curves as a quotient stack (via the the Weierstrass model), and hence it cannot be easily 
extended to other pairs $(g,n)$. 
In Remark \ref{R:bad}, we discuss what  happens if the assumptions of Theorem A  are not satisfied. 

 
Theorem A allows us to compute the Picard group of $\Mgb^k$ and its open substack $\Mg^k$ parametrizing smooth $n$-pointed curves over a (not necessarily algebraically closed) field $k$  if either $\char(k)\neq 2$ or $g\leq 5$.

\begin{ThmB}
Let $k$ be  a field. Assume that either the characteristic of $k$ is different from $2$ or that $g\leq 5$.
\begin{enumerate}
\item \label{C:CorMain1} The Picard group $\Pic(\Mgb^k)$ of $\Mgb^k$ is  a free abelian group of finite rank, generated by the tautological line bundles subject to the unique relations \eqref{E:rela}.
 \item \label{C:CorMain2} The Picard group $\Pic(\Mg^k)$ of $\Mg^k$ is generated by the tautological line bundles $\{\lambda,\psi_1,\ldots, \psi_n\}$ subject to the unique relations
 \begin{equation}\label{E:rela-open}
 \begin{sis}
 &10\lambda=0 & \text{ if }g=2,\\
 & 12 \lambda=0 \quad \text{and }\quad \lambda=\psi_1=\ldots=\psi_n & \text{if }g=1,\\
 & 0=\lambda=\psi_1=\ldots=\psi_n & \text{ if }g=0.
 \end{sis}
 \end{equation}
\end{enumerate}
\end{ThmB}

Note that Theorem B(1)  is the special case of Theorem A(2) where $S=\Spec k$ and it was the original motivation of our work. However, our proof of Theorem B(2)
requires, at least for $g\geq 6$, the study of the Picard group scheme of $\Mgb^{W(\ov k)}$ over  the spectrum of the ring of Witt vectors of the algebraic closure $\ov k$ of $k$.


The above Theorem B was proved for $k=\bbC$ and $g\geq 3$ by Arbarello-Cornalba \cite{AC87}, building upon the topological proof by Harer \cite{Har83, Har85} that $\Pic(\Mg^{\bbC})$ is free of rank $n+1$ if $g\geq 3$. Moreover, the following  special cases of  Theorem B were already known:
\begin{itemize}
\item $\Pic(\ov\cM_{0,n}^k)$ and $\Pic(\cM_{0,n}^k)$ were known  by the work of Keel \cite{Kee92};
\item $\Pic(\ov\cM_{1,1}^k)$ and $\Pic(\cM_{1,1}^k)$ were computed by Fulton-Olsson \cite{FO} (extending to characteristic $2$ and $3$ the previous results of Mumford \cite{Mum65} and Edidin-Graham \cite[Prop. 21]{EG98});
\item $\Pic(\ov\cM_{1,2}^k)$ and $\Pic(\cM_{1,2}^k)$ were computed by Di Lorenzo-Pernice-Vistoli \cite{diLPV} and Inchiostro \cite{Inc} if $\char(k)\neq 2,3$;
\item $\Pic(\cM_{2,0}^k)$ was computed by Vistoli \cite{Vis98} if  $\char(k)\neq 2$
 and Bertin \cite{Ber} if $\char(k)=2$; 
 \item $\Pic(\ov \cM_{2,0}^k)$ was computed by Larson \cite{Lar} if $\char(k)\neq 2,3$;
 \item $\Pic(\cM_{2,1}^k)$ was computed by Pernice \cite{Per} if  $\char(k)\neq 2$;
 \item $\Pic(\ov \cM_{2,1}^k)$ was computed by Di Lorenzo-Pernice-Vistoli \cite{diLPV} if  $\char(k)\neq 2,3$;
\item $\Pic(\cM_{g,0}^k)$ for $g=3,4,5$ were computed by Di Lorenzo \cite{diL}.
\end{itemize}

The above Theorem B was also known for the rational Picard group $\Pic(\Mgb^k)_{\bbQ}$ over an algebraically closed field $k$: the case $k=\bbC$ is easily deduced from the computation of $H^2(\Mgb^{\bbC}, \bbC)$ (and the vanishing of $H^1(\Mgb^{\bbC}, \bbC)$) by Arbarello-Cornalba \cite{AC98},  whose proof uses Hodge theory and the Harer's vanishing theorem \cite{Har86} for the homology of the mapping class group; Moriwaki \cite{Mor01} was able to extend the result to $\Pic(\Mgb^k)_{\bbQ}$ for an arbitrary  $k=\ov k$,  using the case $k=\bbC$ and \'etale cohomology. 

\vspace{0.1cm}
 
As a consequence of Theorem A, and using the well-known fact that the complex analytic first Chern class 
$$c_1^{\an}:\Pic(\Mgb^{\bbC}) \to H^2(\Mgb^{\an},\bbZ)$$
is an isomorphism (see Proposition \ref{P:c1an}), we show in Proposition \ref{P:c1et} that the \'etale first Chern class 
$$c_1^{et}:\Pic(\Mgb^{k})_{\bbZ_l} \to H_{et}^2(\Mgb^{k},\bbZ_l)$$
is an isomorphism for any algebraically closed field $k$ of characteristic different from two and any prime $l\neq \char(k)$.

Moreover, as a consequence of Theorem B, we compute in Proposition \ref{P:Cl} the integral divisor class group $\Cl(\ov M_{g,n}^k)$ of the coarse moduli space $\ov M_{g,n}^k$ of $\Mgb^k$ over an algebraically closed field $k$ of characteristic different from two, generalizing \cite[Prop. 2]{AC87} for $k=\bbC$ and $g\geq 3$.

\vspace{0.1cm}


Let us comment on the strategy and the ingredients in the proof of Theorem A (and also of Theorem B). The proof has two main steps. 


As a first step, we prove in Theorem \ref{T:Pic-taut}  that, if $k$ is an algebraically closed field, then $\Pic(\Mgb^k)$ is a finitely generated (abelian) group and it admits a splitting 
\begin{equation}\label{E:Picsplit}
\Pic(\Mgb^k)=\Pic^{\ta}(\Mgb^k)\oplus \Pic(\Mgb^k)_{\tor},
\end{equation}
where  the (so called) tautological subgroup  $\Pic^{\ta}(\Mgb^k)\subseteq \Pic(\Mgb^k)$, generated  by the tautological line bundles, is torsion-free and such that the relations among the tautological line bundles are generated by the relations \eqref{E:rela}; while  the torsion subgroup $\Pic(\Mgb^k)_{\tor}\subset \Pic(\Mgb^k)$ is a finite (abelian) $p$-group where $p=\char(k)$ (in particular it is trivial when $\char(k)=0$). 
Two crucial ingredients in the proof of \eqref{E:Picsplit} are the fact that   $\Mgb^k$ is algebraically simply connected (see Theorem \ref{T:simply-con}) and the weak Franchetta conjecture (see Theorem \ref{franchetta}, where we extend the proof of  Schr\"oer \cite{Sch03} to $g< 3$) that allows us also to prove that 
\begin{equation}\label{E:nontaut}
\Pic(\Mgb^k)_{\tor} \xrightarrow{\cong} \frac{\Pic(\Mg^k)}{\Pic^{\ta}(\Mg^k)} \xleftarrow{\cong}
 \begin{sis}
&  \frac{\Pic(\cM_g^k)}{\Pic^{\ta}(\cM_g^k)} \: \text{ for } g\geq 2,\\
&  \frac{\Pic(\cM_{1,1}^k)}{\Pic^{\ta}(\cM_{1,1}^k)} \: \text{ for } g=1,\\
\end{sis}
\end{equation}
where  $\Pic^{\ta}(\Mg^k)\subseteq \Pic(\Mg^k)$ is the subgroup generated by the tautological line bundles.


The second step of the proof is the vanishing of the torsion component of the Picard group scheme   (over an algebraically closed field $k$) 
\begin{equation}\label{E:vanPictau}
\PIC^{\tau}_{\Mgb^k/k}=0.
\end{equation}
If $\char(k)=0$, this follows from the  vanishing  $\PIC^{\tau}_{\Mgb^k/k}(k)=\Pic(\Mgb^k)_{\tor}=0$ (see Proposition \ref{P:tau-tors}). 
If $\char(k)=p>0$, we consider  the torsion component $\Pi_{W(k)}^{\tau}: \PIC^{\tau}_{\Mgb^{W(k)}/W(k)}\to \Spec W(k)$ of the Picard group space of $\Mgb^{W(k)}\to \Spec W(k)$, where $W(k)$ is the ring of Witt vectors with residue field $k$, which is a generically trivial group scheme by the above vanishing in characteristic zero. 
From a generalization to algebraic stacks of a result of Raynaud \cite{Ray79} (see Theorem \ref{T:main-Ray} of the Appendix), we deduce that $\Pi_{W(k)}^{\tau}$ is flat if $\char(k)=p>2$ and flat along the zero section  if $\char(k)=p=2$ (see Theorem \ref{T:flat}). This is enough to prove the vanishing \eqref{E:vanPictau} if  either $\char(k)>2$ (see Proposition \ref{P:Pictau}) or  $g\leq 5$ (in this last case we use \eqref{E:nontaut} and the knowledge of $\Pic(\cM_g^k)$ for $g\leq 5$).

\vspace{0.1cm}

The paper is organized as follows. In Section \ref{Sec:prel}, we review some basic properties of the stack $\Mg$ of smooth pointed curves and of the stack $\Mgb$ of stable pointed curves. 
In Section \ref{Sec:PicSpa}, we collect the properties of the Picard group space $\PIC_{\cX/S}$ of a cohomologically flat algebraic stack $f:\cX\to S$, following and expanding upon the work of Brochard \cite{Br1} and \cite{Br2}.
In Section \ref{S:homPhi}, we prove that the tautological homomorphism \eqref{E:Pgnbar}  is well-defined. In Section \ref{S:Picgrp}, we study the Picard group of $\Mgb^k$ over an algebraically closed field $k=\ov k$ and its tautological 
subgroup. In Section \ref{S:conclu}, we prove Theorem A and Theorem B. In Section \ref{S:cons}, we investigate the fist (analytic and \'etale) Chern class of $\Mgb^k$ and the divisor class group of its coarse moduli space $\ov M_{g,n}^k$. In the Appendix \ref{S:App}, we prove a flatness result for $\PIC^{\tau}_{\cX/R}\to \Spec R$, for an algebraic stack $\cX$  proper and cohomologically flat 
 over the spectrum of a discrete valuation ring $R$, extending a result of Raynaud \cite{Ray79} for schemes.    

\begin{convention}
Following \cite{LMB}, an algebraic stack will always be assumed quasi-separated, i.e. with quasi-compact and separated diagonal. In particular, all the schemes and algebraic spaces must be quasi-separated.
\end{convention}

\section{Preliminaries}\label{Sec:prel}


\subsection{The stack $\Mg$ of  smooth pointed curves}

Given two integers $g,n\geq 0$, we will denote by $\Mg$ the stack  whose fiber over a scheme $S$ is the groupoid of families $(f:\cC\to S,\sigma_1,\ldots, \sigma_n)$ of \textbf{smooth $n$-pointed curves of genus $g$} over $S$, i.e. $\cC$ is an algebraic space, $f$ is a smooth and proper morphism, $\{\sigma_1,\ldots, \sigma_n\}$ are sections of $f$ that are fiberwise disjoint and for each geometric point $\ov s\to S$ the fiber $\cC_{\ov s}$ of $f$ over $\ov s$ is a (smooth and projective) connected curve of genus $g$. 


It is well known that the stack $\Mg$ is algebraic, smooth, separated and of finite type over $\Spec \bbZ$, with geometric fibers that are connected (and hence integral) of dimension $3g-3+n$. 
Moreover, the stack $\Mg$ is a DM(=Deligne-Mumford) stack if and only if $3g-3+n>0$.


Given a scheme $S$, we will denote by $\Mg^S\to S$ the base change of $\Mg\to \Spec \bbZ$ via the natural morphism $S\to \Spec \bbZ$. When $S=\Spec R$ for a ring $R$ (resp. $S=\Spec k$ for a field $k$), we will set   $\Mg^R:=\Mg^{\Spec R}$ (resp. $\Mg^k:=\Mg^{\Spec k}$). 

The complex analytic stack associated to the complex algebraic stack $\Mg^{\bbC}$ will be denoted by $\Mg^{\an}$.

\subsection{The stack  $\Mgb$ of  stable pointed curves}

Given two integers $g,n\geq 0$ such that  $3g-3+n>0$ (we call such a pair $(g,n)$ \emph{hyperbolic}),  we will denote by $\Mgb$ the stack whose fiber over a scheme $S$ is the groupoid of families $(f:\cC\to S,\sigma_1,\ldots, \sigma_n)$ of \textbf{stable $n$-pointed curves of genus $g$} over $S$, i.e.  $f$ is a flat and proper morphism, $\{\sigma_1,\ldots, \sigma_n\}$ are sections of $f$ that are fiberwise disjoint, the log dualizing sheaf $\omega_f(\sum_{i=1}^n \sigma_i)$ is an  $f$-relatively ample line bundle and for each geometric point $\ov s\to S$ the fiber $\cC_{\ov s}$ of $f$ over $\ov s$ is a (projective) connected nodal curve of arithmetic genus $g$ such that the points $\{\sigma_1(\ov s), \ldots, \sigma_n(\ov s)\}$ are smooth points of $\cC_{\ov s}$.

The following result is well-known.

\begin{fact}\label{F:Mgnbar}\cite{DM} 
The stack $\Mgb$ is a DM algebraic stack, smooth and proper over $\Spec \bbZ$, with geometric fibers that are connected (and hence integral) of dimension $3g-3+n$. 
In particular, the stack $\Mg$ is an open and dense substack of $\Mgb$. 
\end{fact}

Given a scheme $S$, we will denote by $f_S:\Mgb^S\to S$ the base change of $\Mgb\to \Spec \bbZ$ via the natural morphism $S\to \Spec \bbZ$. When $S=\Spec R$ for a ring $R$ (resp. $S=\Spec k$ for a field $k$), we will set   $\Mgb^R:=\Mgb^{\Spec R}$ (resp. $\Mgb^k:=\Mgb^{\Spec k}$). 





The complex analytic stack (or complex analytic orbifold) associated to the complex algebraic stack $\Mgb^{\bbC}$ will be denoted by $\Mgb^{\an}$.
By a slight abuse of notation, we will also denote by $\Mgb^{\an}$ the DM topological stack (in the sense of \cite[\S 14]{Noo}) naturally associated to $\Mgb^{\bbC}$ (see \cite[\S 20]{Noo}). 
We will need the following well-known result, a proof of which (using Teichm\"uller theory) can be found in \cite[Prop. 1.1]{BP}.

\begin{prop}\label{P:simplycon}
The topological stack $\Mgb^{\an}$ is (topologically) simply connected, i.e. any connected  (topological) cover of $\Mgb^{\an}$ is trivial\footnote{this is equivalent  to the fact that the fundamental group 
$\pi_1(\Mgb^{\an},x)$ of  the topological stack $\Mgb^{\an}$ (with respect to a base point $x$) is trivial, see  \cite[\S 17, \S 18]{Noo}}. 
\end{prop}

\section{The  Picard group space of an algebraic stack}\label{Sec:PicSpa}

The aim of this section (which we believe to be of independent interest) is to collect the properties of the Picard group (algebraic) space of an algebraic stack. Our main sources are the two papers of Brochard \cite{Br1} and \cite{Br2}, but we will also give some complements to the results proved in loc. cit.

If $\cX$ is an algebraic stack, we will denote by $\Pic(\cX)$ the abelian group of invertible sheaves (or line bundles) on $\cX$. 
Given a stack $f:\cX\to S$ over a scheme $S$, the \emph{relative Picard functor} of $\cX\to S$, denoted  by $P_{\cX/S}$, is the functor 
\begin{equation*}\label{E:funcP}
\begin{aligned}
P_{\cX/S}:\Sch_S & \longrightarrow \textbf{Ab}, \\
T & \mapsto \frac{\Pic(\cX_T)}{f_T^*\Pic(T)},
\end{aligned}
\end{equation*}
where $f_T:\cX_T:=\cX\times_S T \to T$ is the base change of $f$ via $T\to S$, and \textbf{Ab} is the category of abelian groups. 
We will denote by $\Pic_{\cX/S}$, and we call it the \emph{Picard sheaf} of $\cX\to S$, the fppf sheaf associated to $P_{\cX/S}$. If $S=\Spec R$ is an affine scheme, then we will set $\Pic_{\cX/R}:=\Pic_{\cX/\Spec R}$.  Note that if $S=\Spec k$ with $k$ an algebraically closed field, then  $\Pic_{\cX/k}(k)=\Pic(\cX)$ (see \cite[Ex. 9.2.3]{Kle05}).

S. Brochard proved in \cite{Br1} the following representability result on $\Pic_{\cX/S}$, generalizing the representability results of Grothendieck, Mumford and Artin (see 
\cite[Sec. 8.2, 8.3]{BLR}, \cite[Sec. 9.4]{Kle05} and the references therein).

\begin{teo}\label{T:repre-Pic}
If  $f:\cX\to S$ is proper and cohomologically flat in degree zero, i.e. $f$ is flat and the natural morphism $\cO_{S}\xrightarrow{\cong} f_*\cO_{\cX}$ is an isomorphism universally, then the Picard sheaf  $\Pic_{\cX/S}$ is represented by a commutative group algebraic space (i.e. a commutative group-like object in the category of algebraic spaces) locally of finite type over $S$, that we  denote by $\PIC_{\cX/S}$ and call it the \emph{Picard group (algebraic) space}. Moreover, $\PIC_{\cX/S}$ commutes with base change, i.e. for any morphism $T\to S$ we have that $\PIC_{\cX_T/T}:=\PIC_{\cX/S}\times_S T$. 
\end{teo}
\begin{proof}
The representability of $\Pic_{\cX/S}$ by an algebraic space follows by \cite[Thm. 1.1, Prop. 2.3.3]{Br1}. The fact that $\PIC_{\cX/S}$ is an algebraic commutative group space follows from the fact that its functor of points, i.e. $\Pic_{\cX/S}$, take values in the category $\AbGr$ of abelian groups.
The fact that $\PIC_{\cX/S}$ is locally of finite type over $S$ is proved as in \cite[Prop. 9.4.17]{Kle05}. The fact that $\PIC_{\cX/S}$ commutes with base change follows from \cite[Ex. 9.4.4]{Kle05}.
\end{proof}

Note that if $f:\cX\to S$ is proper and flat with geometrically reduced and geometrically connected fibers, then $f$ is cohomologically flat in degree $0$ (see \cite[Ex. 9.3.11]{Kle05}). 


Assume now that $\cX\to \Spec k$ is a proper and cohomologically flat stack over a field $k$. Then $\PIC_{\cX/k}$ is a commutative group scheme locally of finite type over $k$, as it follows from Theorem \ref{T:repre-Pic} and the fact that every group  algebraic space locally of finite type over a field is a group scheme by \cite[Lemma 4.2, p. 43]{Art69}. Hence we can call $\PIC_{\cX/k}$ the \emph{Picard group scheme}  of $\cX/k$. The smoothness and the dimension of $\PIC_{\cX/k}$ is governed by the first and second cohomology groups of the structure sheaf of $\cX$. 

\begin{prop}\label{P:smPic}
Let $\cX\to \Spec k$ be a proper and cohomologically flat stack over a field $k$. 
\begin{enumerate}[(i)]
\item \label{P:smPic1} We have 
$$\dim  \PIC_{\cX/k}\leq T_e  \PIC_{\cX/k}=H^1(\cX, \cO_{\cX}),$$
where $T_e  \PIC_{\cX/k}$ is the tangent space of  $\PIC_{\cX/k}$ at its identity element  $e$. 

Moreover, equality holds in the above inequality if and only if $\PIC_{\cX/k}$ is smooth.
\item \label{P:smPic2} If $\char(k)=0$ or $H^2(\cX, \cO_{\cX})=0$, then $\PIC_{\cX/k}$ is smooth.
\end{enumerate}
\end{prop}
\begin{proof}
This follows from the following facts:
\begin{itemize}
\item  the tangent (resp. an obstruction) space to the deformation functor of a line bundle on $\cX$ is equal to $H^1(\cX, \cO_{\cX})$ (resp. $H^2(\cX, \cO_{\cX})=0$), see \cite[Thm. 3.1.3, Thm. 4.1.3]{Br1};
\item a group scheme $G$ locally of finite type over a field $k$ is smooth if and only if it is smooth at its identity element $e\in G(K)$;
\item if $\char(k)=0$, then every  group scheme $G$ locally of finite type over $k$ is smooth (Cartier's theorem).
\end{itemize}
\end{proof}
The Picard group scheme $\PIC_{\cX/k}$ (of a proper and cohomologically flat stack $\cX\to \Spec k$ over a field $k$) admits two notable group subschemes
\begin{equation}\label{E:Pico-tau}
\PIC^o_{\cX/k}\subset \PIC^{\tau}_{\cX/k}\subset \PIC_{\cX/k},
\end{equation}
where $\PIC^o_{\cX/k}$ is the \emph{connected component} of $\PIC_{\cX/k}$ containing the identity element  and $\PIC^{\tau}_{\cX/k}$ is the \emph{torsion component}  of $\PIC_{\cX/k}$ containing the identity element, i.e. 
$$\PIC^{\tau}_{\cX/k}=\bigcup_{n>0} \phi_n^{-1}(\PIC^{o}_{\cX/k}),$$
where $\phi_n:\PIC_{\cX/k}\to \PIC_{\cX/k}$ is the $n$-th power map homomorphism.  Note that:
\begin{itemize}
\item  $\PIC^o_{\cX/k}$ is an open and closed subscheme; it is geometrically irreducible and of finite type over $k$; and forming it  commutes with field extensions (see   \cite[Lemma 9.5.1]{Kle05}).
\item $\PIC^{\tau}_{\cX/k}$ is an open and closed subscheme of $\PIC_{\cX/k}$; it is the biggest group subscheme of $\PIC_{\cX/k}$ that is of finite type over $k$; and forming it  commutes with field extensions (see \cite[Lemma 9.6.9, Ex. 9.6.10]{Kle05} using \cite[Thm. 3.3.3(2)]{Br2}).

\end{itemize}


The geometric meaning of the above two group subschemes over algebraically closed fields is clarified by the following result.

\begin{prop}\label{P:geo-Pico}
Let $\cX\to \Spec k$ be a proper and cohomologically flat stack over an algebraically closed field $k$. Let $\cL$ be a line bundle on $\cX$ and let $[\cL]\in \PIC_{\cX/k}(k)$ be its class. 
\begin{enumerate}[(i)]
\item \label{P:geo-Pico1} $[\cL]\in \PIC^o_{\cX/k}(k)$ if and only if $\cL$  is algebraic equivalent to  the trivial line bundle $\cO_{\cX}$, i.e. there exists a connected scheme $T$ of finite type over $k$,
a line bundle $\bbL$ on $\cX\times_k T$ and two points $t_0,t_1\in T(k)$  such that $\bbL_{|\cX\times \{t_0\}}=\cL$ and  $\bbL_{|\cX\times \{t_1\}}=\cO_{\cX}$.
\item \label{P:geo-Pico2} $[\cL]\in \PIC^{\tau}_{\cX/k}(k)$ if and only if $\cL$  is numerically trivial, i.e. for any integral proper $k$-curve $C$ contained in $\cX$ we have that $\deg(\cL_{|C})=0$.  
\end{enumerate}
\end{prop}
\begin{proof}
Part \eqref{P:geo-Pico1} can be proved as in \cite[Prop. 9.5.10]{Kle05} (as observed in \cite[Thm. 4.2.3]{Br1}).

Part \eqref{P:geo-Pico2}: the case when $\cX$ is a projective scheme over $k$ is well known (see \cite[Thm. 9.6.3, Ex. 9.6.11]{Kle05}). Let us now prove the general case. 

If $[\cL]\in \PIC^{\tau}_{\cX/k}(k)$ then, by definition of $\PIC^{\tau}_{\cX/k}$ and part \eqref{P:geo-Pico1}, we have that $\cL^n$ is algebraically equivalent to $\cO_{\cX}$ for some $n\geq 1$. 
This implies that, for any integral proper $k$-curve $C\subset \cX$, $\cL_{|C}^n$ is algebraically equivalent to $\cO_{C}$, which in turn implies that $\deg(\cL_{|C})=0$ since the degree of line bundles on a proper integral curve is invariant under algebraic equivalence. 

Conversely, assume that $\cL$ is numerically trivial. By Chow's lemma for stacks (see \cite[Thm. 1.1]{Ols05}), there exists a projective $k$-scheme $X$ endowed with a surjective (proper) morphism $X\to \cX$. By the projection formula, we have that $f^*(\cL)$ is numerically trivial on $C$. Since the statement we want to prove is true for $\PIC_{X/k}$ (as recalled above), we deduce that $[f^*(\cL)]=f^*([\cL])\in \PIC^{\tau}_{X/k}(k)$ where $f^*:\PIC_{\cX/k}\to \PIC_{X/k}$ is the morphism induced by the pull-back.  Since $(f^*)^{-1}(\PIC^{\tau}_{X/k})=\PIC^{\tau}_{\cX/k}$ by \cite[Lemma 3.3.2]{Br2}, we conclude that $[\cL]\in \PIC^{\tau}_{\cX/k}(k)$. 
\end{proof}

The connected and torsion components of the Picard group scheme of $\cX\to \Spec k$ are quasi-projective, and projective if $\cX$ is geometrically normal.

\begin{prop}\label{P:Pico-tau}
Let $\cX\to \Spec k$ be a proper and cohomologically flat stack over a field $k$. 
\begin{enumerate}[(i)]
\item \label{P:Pico-tau1} The schemes $\PIC^o_{\cX/k}$ and $\PIC^{\tau}_{\cX/k}$ are quasi-projective  over $k$. 
\item \label{P:Pico-tau2} If $\cX$ is  geometrically normal, then $\PIC^o_{\cX/k}$ and $\PIC^{\tau}_{\cX/k}$ are projective schemes over $k$. 
\end{enumerate}
\end{prop}
\begin{proof}
The scheme $\PIC^o_{\cX/k}$ is quasi-projective over $k$ by \cite[Cor. 2.3.7]{Br2}. If, furthermore, $\cX$ is geometrically normal, then  $\PIC^o_{\cX/k}$ is proper over $k$ by \cite[Thm. 4.3.1]{Br1},,
and hence projective over $k$. This proves the result for $\PIC^o_{\cX/k}$. 

The analogue result for $\PIC^{\tau}_{\cX/k}$ follows since $\PIC^o_{\cX/k}$ is the connected component of $\PIC^{\tau}_{\cX/k}$ containing the identity and 
$\PIC^{\tau}_{\cX/k}$ is of finite type by \cite[Thm. 3.3.3(2)]{Br2} and hence it has finitely many connected components. 
\end{proof}


The quotients of the Picard group scheme $\PIC_{\cX/k}$ (of a proper and cohomologically flat stack $\cX\to \Spec k$ over a field $k$) by the two subgroups of \eqref{E:Pico-tau}
\begin{equation}\label{E:NS-quot}
\NSb_{\cX/k}:=\frac{\PIC_{\cX/k}}{\PIC^o_{\cX/k}}\twoheadrightarrow \NSb^{\nu}_{\cX/k}:=\frac{\PIC_{\cX/k}}{\PIC^{\tau}_{\cX/k}}
\end{equation}
are commutative group schemes \'etale and locally of finite type over $k$, which are called, respectively,  the \emph{Neron-Severi}  and \emph{numerical Neron-Severi}  group schemes of $\cX/k$.
Note that the formation of $\NSb_{\cX/k}$ and $\NSb^{\nu}_{\cX/k}$ commutes with field extensions. 

If $\cX$ is a proper and cohomologically flat stack over an algebraically closed field $k=\ov k$  then we set 
$$
\begin{sis}
\Pic^o(\cX):=\PIC^o_{\cX/k}(k) \:  & \text{ and } \: \Pic^{\tau}(\cX):=\PIC^{\tau}_{\cX/k}(k), \\
\NS(\cX):=\NSb_{\cX/k}(k)=\frac{\Pic(\cX)}{\Pic^o(\cX)} \: & \text{ and } \NS^{\nu}(\cX):=\NSb^{\nu}_{\cX/k}(k)=\frac{\Pic(\cX)}{\Pic^{\tau}(\cX)}.
\end{sis}
$$

\begin{prop}\label{P:finit-NS}
Let $\cX\to \Spec k$ be a proper and cohomologically flat stack over an algebraically closed field $k$. 
Then $\NS(\cX)$  is a finitely generated abelian group such that its torsion subgroup is equal to 
\begin{equation}\label{E:torNS}
\NS(\cX)_{\tor}=\frac{\Pic^{\tau}(\cX)}{\Pic^o(\cX)},
\end{equation}
and its torsion-free quotient is equal to 
\begin{equation}\label{E:tfNS}
\NS(\cX)_{\tf}=\NS^{\nu}(\cX). 
\end{equation}
\end{prop} 
\begin{proof}
The fact that $\NS(\cX)=\NSb_{\cX/k}(k)$ is a finitely generated abelian group follows from \cite[Thm. 3.4.1]{Br2}. The other assertions follows immediately from the definition of $\PIC^{\tau}_{\cX/k}$.
\end{proof}

Now, going back to the general case of a proper and cohomologically flat stack $\cX\to S$ over a general scheme $S$, we can consider the following two locally constructible subsets of $\PIC_{\cX/S}$ (see \cite[VI, Thm. 1.1]{FGA})
\begin{equation}\label{E:Pico-tauS}
\PIC^o_{\cX/S}:=\bigcup_{s\in S} \PIC^o_{\cX_s/k(s)}\subset \PIC^{\tau}_{\cX/S}:=\bigcup_{s\in S} \PIC^{\tau}_{\cX_s/k(s)}\subset \PIC_{\cX/k}.
\end{equation}
While the openess of $\PIC^o_{\cX/S}$ is a delicate issue (see \cite[Prop. 9.5.20]{Kle05} and \cite[Prop. 4.2.10]{Br1} for some criteria and \cite[Ex. 3.3.4]{Br2} for a counterexample), it turns out that $\PIC^{\tau}_{\cX/S}$ is open if $f$ is  finitely presented and it has many nice properties, as we now show.

\begin{teo}\label{T:Pictau}
Let $f: \cX\to S$ be a proper and cohomologically flat stack over a scheme $S$ and assume that  $f$ is finitely presented. 
\begin{enumerate}[(i)]
\item \label{T:Pictau1} 
$\PIC^{\tau}_{\cX/S}$  is an open group (algebraic) subspace of  $\PIC_{\cX/S}$, which is of finite type over $S$ and whose formation commutes with base change. 
\item \label{T:Pictau2} Let $s$ be a point of $S$ such that $H^2(\cX_s,\cO_{\cX_s})=0$. Then there exists an open neighboorhood of $s$ in $S$ over which $\PIC_{\cX/S}$ (and hence also $\PIC^{\tau}_{\cX/S}$) is smooth.  
\item \label{T:Pictau3} If the fibers of $f$ are geometrically normal then 
\begin{enumerate}[(a)]
\item \label{T:Pictau3a} $\PIC_{\cX/S}$ (and hence also $\PIC^{\tau}_{\cX/S}$) is separated over $S$.
\item  \label{T:Pictau3b} $\PIC_{\cX/S}$ (and hence also  $\PIC^{\tau}_{\cX/S}$) is equidimensional over $S$.
\item  \label{T:Pictau3c}  $\PIC^{\tau}_{\cX/S}$ is closed in $\PIC_{\cX/S}$.
\end{enumerate}
\item \label{T:Pictau4} If the fibers of $f$ are smooth (or, equivalently, if $f$ is smooth), then 
\begin{enumerate}
\item  \label{T:Pictau4a} $\PIC_{\cX/S}\to S$ satisfies the valuative criterion of properness.
\item \label{T:Pictau4b} $\PIC^{\tau}_{\cX/S}$ is proper over $S$. 
\end{enumerate}
\item \label{T:Pictau5} If $S$ is Noetherian, then 
\begin{itemize}
\item the number of geometric connected components of the fibers of $\PIC^{\tau}_{\cX/S}\to S$
 is bounded.
\item the ranks of the free abelian groups $\NS^{\nu}(\cX_{\ov{s}})$ (as $s$ varies among the points of $S$) are bounded.
\end{itemize}
\end{enumerate}
\end{teo}
Note that, under the assumptions of \eqref{T:Pictau4}, the morphism $\PIC_{\cX/S}\to S$ is not proper, in general, because it is not quasi-compact.
\begin{proof}
Part \eqref{T:Pictau1}: the fact that $\PIC^{\tau}_{\cX/S}\subset \PIC_{\cX/S}$ is open and that $\PIC^{\tau}_{\cX/S}$ is of finite type over $S$ follows from \cite[Thm. 3.3.3]{Br2}; the fact that the formation of $\PIC^{\tau}_{\cX/S}$ commutes with base changed follows from the same property of $\PIC_{\cX/S}$ together with the definition of $\PIC^{\tau}_{\cX/S}$ and the fact that the formation of each $\PIC^{\tau}_{\cX_s/k(s)}$ (for $s\in S$) commutes with field extensions. 


Part \eqref{T:Pictau2} can be  proved as in \cite[Prop. 9.5.19]{Kle05}, using that $H^2(\cX_s,\cO_{\cX_s})$ is an obstruction space for the deformation functor of a line bundle on $\cX_s$ (see \cite[Thm. 3.1.3]{Br1}).

Part \eqref{T:Pictau3a} follows from \cite[Prop. 3.2.5]{Br1}.

Part \eqref{T:Pictau3b} is proved as in \cite[VI, Cor. 2.7]{FGA} using that the  fibers $\PIC_{\cX_s/k(s)}$ do not contain additive components by Proposition \ref{P:Pico-tau}\eqref{P:Pico-tau2}. 

Part \eqref{T:Pictau3c} is proved as in \cite[VI, Cor. 2.3]{FGA} using Proposition \ref{P:Pico-tau}\eqref{P:Pico-tau2} and \cite[VI, Thm. 1.1]{FGA}.

Part \eqref{T:Pictau4a}: $\PIC_{\cX/S}\to S$ satisfies the uniqueness part of the valuative criterion because $\PIC_{\cX/S}$ is separated over $S$, by part \eqref{T:Pictau3a}. The fact that $\PIC_{\cX/S}\to S$ satisfies the existence part of the valuative criterion, under the assumption that $f$ is smooth, is proved as in \cite[VI, Thm. 2.1]{FGA}.

Part \eqref{T:Pictau4b}: $\PIC^{\tau}_{\cX/S}$ satisfies the valuative criterion of properness since $\PIC_{\cX/S}\to S$ does by \eqref{T:Pictau4a} and $\PIC^{\tau}_{\cX/S}\subseteq \PIC_{\cX/S}$ is a closed embedding (hence proper) by \eqref{T:Pictau4b}. We conclude that $\PIC^{\tau}_{\cX/S}$ is proper over $S$ since it is of finite type over $S$ by \eqref{T:Pictau1}.

Part \eqref{T:Pictau5} follows from \cite[Thm. 3.4.1]{Br2}.
\end{proof}

Note that  the ranks of the free abelian groups $\NS^{\nu}(\cX_{\ov{s}})$ can jump by specializations of the points of $S$ and the function $s\mapsto \rk \NS^{\nu}(\cX_{\ov{s}})$ is not constructible (see \cite[p. 235]{BLR} for an explicit example and \cite{MP} for a detailed discussion on the specialization of Neron-Severi groups).

\section{The tautological homomorphism}\label{S:homPhi}

The aim of this section is to construct the tautological  homomorphism $\ov P_{g,n}$ of \eqref{E:Pgnbar}. With this in mind,  let us  introduce the following abelian groups.

\begin{defin}\label{D:abgrps}
Let $\Gamma_{g,n}$ the free abelian group generated by 
$$\un{\lambda}, \: \un{\delta_{\irr}} \: \text{ and } \{\un{\delta_{a,A}}\: : 0\leq a \leq g, \emptyset \subseteq A\subseteq [n] \text{ with } (a,A)\neq (0,\emptyset), (g, [n])\}.$$

Let $R_{g,n}$ the subgroup of $\Gamma_{g,n}$ generated by the relations \eqref{E:rela}, so that $\Lambda_{g,n}=\Gamma_{g,n}/R_{g,n}$.
\end{defin}

Consider the universal family $\pi:\Cgb\to \Mgb$ and denote by $\{\sigma_1,\ldots, \sigma_n\}$ the $n$ universal sections. We define the following homomorphism 
\begin{equation}\label{E:Pgn}
\begin{aligned}
P_{g,n}: \Gamma_{g,n} & \longrightarrow \Pic(\Mgb), \\
\un{\lambda} & \mapsto \lambda=\det \left((R\pi)_*(\omega_{\pi})\right), \\
\un{\delta_{\irr}} & \mapsto \delta_{\irr}=\cO(\Delta_{\irr}), \\
\un{\delta_{a,A}} & \mapsto 
\begin{cases} 
\delta_{0,\{i\}}:=-\psi_i=-\sigma_i^*(\omega_{\pi})& \text{ if } (a,A)=(0,\{i\}) \text{ or } (a,A)=(g,\{i\}^c), \\
\delta_{a,A}=\cO(\Delta_{a,A}) & \text{ otherwise,} 
\end{cases}  
\end{aligned}
\end{equation}
where  $\Delta_{\irr}$ and $\Delta_{a,A}$  are the boundary (Cartier) divisors  of $\Mgb$ whose generic points are, respectively, curves with one non-separating node or curves with one separating node whose normalization is a disjoint union of a curve of genus $a$ containing the marked points  in $A$ and a curve of genus $g-a$ containing the marked points in $A^c$ (see \cite[Def. 3.8]{Knu2}).  
The notation $\delta_{0,\{i\}}:=-\psi_i$ is taken from  \cite[Sec. 2]{GKM}, and it is very convenient since it allows to present in a more compact format several of the  formulas we will encounter below. 


\begin{prop}\label{P:relPgnS}
The subgroup $R_{g,n}\subset \Gamma_{g,n}$ is contained in the kernel of $P_{g,n}$. In particular, the homomorphism $P_{g,n}$ factors through a homomorphism 
\begin{equation*}\label{E:Pgnb}
\ov P_{g,n}:\Lambda_{g,n}\to \Pic(\Mgb).
\end{equation*}
\end{prop}
\begin{proof}
We will prove that each of the relations of \eqref{E:rela} holds true for the tautological line bundles on $\Pic(\Mgb)$.

First of all, the relation ${\delta_{a,A}}={\delta_{g-a,A^c}}$ holds true by definition if $(a,A)=(0,\{i\}), (g,\{i\}^c)$, and because $\Delta_{a,A}=\Delta_{g-a,A^c}$ in the other cases. In order to look at the other relations, we will distinguish the cases $g=0,1,2$.

\un{I Case:} $g=0$.
\begin{itemize}
\item The relation $\delta_{\irr}=0$ hold true because $\Delta_{\irr}=\emptyset$ if $g=0$. 
\item The relation $\lambda=0$ holds true because $\pi_*(\omega_{\pi})=0$ and $(R^1\pi)_*(\omega_{\pi})=\cO_{\Mgb}$. 
\item The relation 
$$\sum_{\substack{z\in A\\x,y\not\in A}}\delta_{0,A}=0$$ 
holds true for $n=3$ because $\ov{\cM}_{0,3}=\Spec \bbZ$ and hence $\delta_{0,\{z\}}=-\psi_z=0$. The case of $n\geq 4$ follows by pull-back through the morphism 
$$\zeta_{x,y,z}:\ov\cM_{0,n}\to \ov{\cM}_{0,3}$$
that forgets the marked points in $[n]\setminus\{x,y,z\}$ (and then it stabilizes), and using the pull-back formulas of \cite[Lemma 3.1]{AC98}.
\item The relation
$$\sum_{\substack{p,q\in A \\r,s\not\in A}}\delta_{0,A}= \sum_{\substack{p,r\in A \\q,s\not\in A}}\delta_{0,A}$$ 
holds true for $n=4$ because $\ov{\cM}_{0,4}=\bbP^1_{\bbZ}$ and the three boundary divisors  (which are identified with $0_\bbZ$, $\infty_\bbZ$ and $1_\bbZ$) are linearly equivalent. 
The case of $n\geq 5$ follows by pull-back through the morphism 
$$\zeta_{p,q,r,s}:\ov\cM_{0,n}\to \ov{\cM}_{0,4}$$
that forgets the marked points in $[n]\setminus\{p,q,r,s\}$ (and then it stabilizes), and using the pull-back formulas of \cite[Lemma 3.1]{AC98}.

\end{itemize}

\un{II Case:} $g=1$.
 The relations 
$$\begin{sis}
& 12\lambda-\delta_{\irr}=0, \\
& \lambda+\sum_{p\in A} \delta_{0,A}=0
\end{sis}
$$
 hold in $\Pic(\ov\cM_{1,1})$, as it is shown in \cite[Chap. XIII, Example (7.11)]{GAC2} (whose proof works over an arbitrary base and not only over $\bbC$). The case of $n\geq 2$ follows by pull-back through the morphism 
$$\zeta_{p}:\ov\cM_{1,n}\to \ov{\cM}_{1,1}$$
that forgets the marked points in $[n]\setminus\{p\}$ (and then it stabilizes), and using the pull-back formulas of \cite[Lemma 3.1]{AC98}.

\un{III Case:} $g=2$.

We have to prove that the relation
\begin{equation}\label{E:relg2}
10\lambda-\delta_{\irr}-2\sum_{A\subseteq [n]} \delta_{1,A}=0
\end{equation}
holds true in $\Pic(\ov\cM_{2,n})$. We will prove it in three steps.

\begin{itemize}
\item Step I: The relation  
\begin{equation}\label{E:relg20}
10\lambda-\delta_{\irr}-2\delta_{1}=0
\end{equation}
 holds true in $\Pic(\ov\cM_{2}^k)$ for any algebraically closed field $k$\footnote{This was proved in \cite{Cor07} if $\char(k)=0$.}. 
 
 Indeed, since the stack $\ov\cM_{2}^k$ is regular, we have an exact sequence (see \cite[Prop. 1.9]{PTT15})
 $$
 0\to \langle \delta_{\irr}, \delta_{1}\rangle \to \Pic(\ov\cM_{2}^k)\to \Pic(\cM_{2}^k)\to 0.
 $$
We know that $\Pic(\cM_{2}^k)$ is a cyclic group of order $10$ generated by $\lambda$ (see \cite{Vis98} if $\char(k)\neq 2$ and \cite{Ber} if $\char(k)=2$). Therefore, we must have that 
\begin{equation}\label{E:relg2NEW}
10\lambda=a\delta_{\irr}+b\delta_1\in \Pic(\ov\cM_{2}^k)\quad \text{ for certain } a,b\in \bbZ.
\end{equation}
 By the structure of  $\Pic(\ov\cM_{2}^k)_{\bbQ}$ (see \cite{AC98} if $\char(k)=0$ and \cite{Mor01} if $\char(k)=p>0$), it follows that $\delta_{\irr}$ and $\delta_1$ are linearly independent, and that the relation \eqref{E:relg20} holds true in  $\Pic(\ov\cM_{2}^k)_{\bbQ}$. Hence, we deduce that  the relation \eqref{E:relg2NEW} must coincide with  \eqref{E:relg20}, and we are done. 

\item Step II: The relation  \eqref{E:relg20} holds true in $\Pic(\ov\cM_{2})$. 

 The line bundle $\cL:=10\lambda-\delta_{\irr}-2\delta_{1}\in \Pic(\ov\cM_2)$ is trivial on the geometric fibers of $f:\ov\cM_{2}\to \Spec \bbZ$ by the Step I. Hence, by the seesaw principle applied to the proper morphism $f$ with geometrically integral fibers (which holds by the same proof of \cite[Cor. B.9]{Fri18} replacing \cite[Prop. B.4]{Fri18} with \cite[Cor. A.2.4]{Br2}), we get that $\cL=f^*(\cM)$ for a certain line bundle $\cM$ on $\Spec \bbZ$. However, $\Pic(\Spec \bbZ)=0$ which implies that $\cL$ is trivial. 

\item Step III: The relation  \eqref{E:relg2} holds true in $\Pic(\ov\cM_{2,n})$.

Consider  the morphism 
$$\zeta:\ov\cM_{2,n}\to \ov{\cM}_{2}$$
that forgets all the marked points  (and then it stabilizes). Using the pull-back formulas of \cite[Lemma 3.1]{AC98}, it follows that the pull-back of the relation \eqref{E:relg20} in $\Pic(\ov\cM_2)$ is equal to  the relation \eqref{E:relg2} in $\Pic(\ov\cM_{2,n})$ and we are done. 

\end{itemize}

\end{proof}

\section{The Picard group of $\Mgb^k$}\label{S:Picgrp}

\emph{Throughout this section, we assume that $k=\ov k$ is an algebraically closed field of characteristic $0$ or $p>0$.}

The aim of this section is to study the Picard group of $\Mgb^k$ and of $\Mg^k$. First we made the following

\begin{defin}\label{D:tautPic}[Tautological Picard groups]
\noindent 
\begin{enumerate}[(i)]
\item The \emph{tautological Picard group of $\Mgb^k$}, denoted by $\Pic^{\ta}(\Mgb^k)$, is the subgroup of the Picard group $\Pic(\Mgb^k)$ generated by the tautological line bundles, i.e. the Hodge line bundle $\lambda$, the cotangent line bundles $\{\psi_i\}_{i=1}^n$, and the boundary divisors $\delta_{\irr}$ and 
$$\{\delta_{a,A}\: : 0\leq a \leq g, \emptyset \subseteq A\subseteq [n] \text{ with } |A|\geq 2 \: \text { if } a=0 \: \text{ and } \: |A|\leq g-2\:  \text{ if } a=g\},$$
with the identification $\delta_{a,A}=\delta_{g-a,A^c}$. For convenience in the formulas that follow, we will set $\delta_{0,\{i\}}=\delta_{g,\{i\}^c}:=\psi_i$ for any $1\leq i \leq n$ (following the notation of \cite[Sec. 2]{GKM}).
\item The \emph{tautological Picard group of $\Mg^k$}, denoted by $\Pic^{\ta}(\Mg^k)$, is the subgroup of the Picard group $\Pic(\Mg^k)$ generated by the tautological line bundles, i.e. the Hodge line bundle $\lambda$ and the cotangent line bundles $\{\psi_i\}_{i=1}^n$. 
\end{enumerate}
\end{defin}

We collect in the following Theorem all the results that will be proved in this section.  

\begin{teo}\label{T:Pic-taut}
\noindent 
\begin{enumerate}
\item \label{T:Pic-taut1} The tautological Picard group $\Pic^{\ta}(\Mgb^k)$ of $\Mgb^k$ is a free abelian group and the subgroup of relations among the tautological line bundles is generated by the relations  given in \eqref{E:rela}. 
\item \label{T:Pic-taut2} The tautological Picard group $\Pic^{\ta}(\Mg^k)$ of $\Mg^k$ is equal to 
\begin{equation}\label{E:ex-Pic-op}
\Pic^{\ta}(\Mg^k)=
\begin{cases}
\bbZ\langle \lambda\rangle \oplus \bbZ\langle \psi_1\rangle\oplus \cdots \oplus \bbZ\langle \psi_n\rangle \cong \bbZ^{n+1}& \text{ if } g\geq 3, \\
\frac{\bbZ}{10\bbZ} \langle \lambda\rangle \oplus \bbZ\langle \psi_1\rangle\oplus \cdots \oplus \bbZ\langle \psi_n\rangle \cong \frac{\bbZ}{10\bbZ}\oplus \bbZ^{n}& \text{ if } g=2, \\
\frac{\bbZ}{12\bbZ} \: \text{ generated by } \lambda=\psi_1=\ldots=\psi_n & \text{ if } g=1, \\
0 & \text{ if } g=0.
\end{cases}
\end{equation}
Moreover, the restriction homomorphism $\res_{g,n}:\Pic(\Mgb^k)\to \Pic(\Mg^k)$ induces an isomorphism 
\begin{equation}\label{E:res-iso}
[\res_{g,n}]:   \frac{\Pic(\Mgb^k)}{\Pic^{\ta}(\Mgb^k)}\xrightarrow{\cong} \frac{\Pic(\Mg^k)}{\Pic^{\ta}(\Mg^k)}.  
\end{equation}
\item \label{T:Pic-taut3} Let $F:\Mg^k\to \cM_g^k$ for $g\geq 2$ (resp. $F:\cM_{1,n}^k\to \cM_{1,1}^k$ for $g=1$) be the forgetful morphism. The pull-back map $F^*$ on the Picard groups is injective and it induces isomorphisms 
\begin{equation}\label{E:pull-iso}
\begin{sis}
& [F^*]: \frac{\Pic(\cM_g^k)}{\Pic^{\ta}(\cM_g^k)}\xrightarrow{\cong} \frac{\Pic(\Mg^k)}{\Pic^{\ta}(\Mg^k)} \: \text{ for } g\geq 2,\\
& [F^*]: \frac{\Pic(\cM_{1,1}^k)}{\Pic^{\ta}(\cM_{1,1}^k)}\xrightarrow{\cong} \frac{\Pic(\cM_{1,n}^k)}{\Pic^{\ta}(\cM_{1,n}^k)} \: \text{ for } g=1.\\
\end{sis}
\end{equation}
\item \label{T:Pic-taut4} 
There is a canonical splitting 
\begin{equation}\label{E:splitPic}
\Pic(\Mgb^k)=\Pic^{\ta}(\Mgb^k)\oplus \Pic(\Mgb^k)_{\tor}.
\end{equation}
Moreover, we have that
\begin{enumerate}
\item \label{T:Pic-taut4a}  $\Pic(\Mgb^k)_{\tor}$ is equal to zero if either $\char(k)=0$  or $g\leq 5$. 
\item \label{T:Pic-taut4b} $\Pic(\Mgb^k)_{\tor}$ is a finite $p$-group if $\char(k)=p>0$. 
\item \label{T:Pic-taut4c} If $g\geq 2$, then the pull-back along the forgetful-stabilization morphism $\ov F:\Mgb^k\to \ov{\cM}^k_g$ induces an isomorphism 
\begin{equation}\label{E:tor-pull}
\ov F^*: \Pic(\ov{\cM}^k_g)_{\tor}\xrightarrow{\cong} \Pic(\Mgb^k)_{\tor}.
\end{equation}
\end{enumerate}
\end{enumerate}
\end{teo}

The above Theorem is well-known if $g=0$ due to the work of Keel \cite{Kee92}. 
 Therefore, we will sometimes assume in what follows that $g\geq 1$. 
 
 \begin{rmk}\label{R:rkPic}
 It follows from the above Theorem (see also \cite{Kee92} for $g=0$) that 
 $$
 \rk \Pic^{\ta}(\Mgb^k)=
 \begin{cases} 
1+ \left\lceil \frac{2^n(g+1)}{2}\right\rceil  & \text{ if } g\geq 3,\\ 
  \left\lceil \frac{2^n(g+1)}{2}\right\rceil & \text{ if } g=2,\\ 
  2^n-n & \text{ if } g=1,\\ 
  2^{n-1}-\binom{n}{2}-1 & \text{ if } g=0.\\ 
 \end{cases}
$$ 
 \end{rmk}

We subdivide the proof of the above Theorem  into four separate subsections, one for each item.

\subsection{Tautological Picard group $\Pic^{\ta}(\Mgb^k)$}\label{S:item1}

The aim of this subsection is to prove Theorem \ref{T:Pic-taut}\eqref{T:Pic-taut1}. 

By the definition of the tautological subgroup $\Pic^{\ta}(\Mgb^k)\subseteq \Pic(\Mgb^k)$, the homomorphism \eqref{E:Pgn} gives rise to  a surjective homomorphism 
\begin{equation}\label{E:Pgnk}
P_{g,n}(k): \Gamma_{g,n}  \twoheadrightarrow \Pic^{\ta}(\Mgb^k),
\end{equation}
whose kernel contains $R_{g,n}$ by Proposition \ref{P:relPgnS}.





\begin{proof}[Proof of Theorem \ref{T:Pic-taut}\eqref{T:Pic-taut1}]
Consider the following commutative diagram 
\begin{equation}\label{E:diag-Pic}
\begin{tikzcd}        
\ker(P_{g,n}(k)) \arrow[hook]{r} \arrow[hook]{d}  & \Gamma_{g,n}\cap (R_{g,n})_{\bbQ}   \arrow[hook]{r} \ar[hook]{d}& (R_{g,n})_{\bbQ} \arrow[hook]{d} \\
\Gamma_{g,n} \ar[equal]{r} \arrow[twoheadrightarrow]{d}{P_{g,n}(k)} & \Gamma_{g,n}  \arrow[hook]{r}   \arrow[twoheadrightarrow]{d}& (\Gamma_{g,n})_{\bbQ}  \arrow[twoheadrightarrow]{d}\\
\Pic^{\ta}(\Mgb^k) \arrow[twoheadrightarrow]{r} &  \Pic^{\ta}(\Mgb^k)_{\tf}\arrow[hook]{r} & \Pic^{\ta}(\Mgb^k)_{\bbQ}
\end{tikzcd}
\end{equation}
The vertical columns in the above diagram are all exact: the left vertical column  is exact by definition;  the right vertical column is exact by the description of $ \Pic^{\ta}(\Mgb^k)_{\bbQ}$ (which indeed coincide with  $\Pic(\Mgb^k)_{\bbQ}$) obtained by \cite[Thm. 2.2]{AC98} in characteristic zero and \cite{Mor01} in positive characteristic; the middle vertical column is exact because of the  inclusion $\Pic^{\ta}(\Mgb^k)_{\tf}\hookrightarrow \Pic^{\ta}(\Mgb^k)_{\bbQ}$. 

Now observe that
\begin{equation}\label{E:equ-Rgn}
\Gamma_{g,n}\cap (R_{g,n})_{\bbQ}=R_{g,n}\subseteq \ker(P_{g,n}(k)),
\end{equation}
where the first equality is a straightforward computation using the generators \eqref{E:rela} of $R_{g,n}$ and the second one follows from Proposition \ref{P:relPgnS}. 

From the commutative diagram \eqref{E:diag-Pic} with exact columns and the equalities \eqref{E:equ-Rgn}, we get 
\begin{equation}\label{E:pres-Pict}
\Pic^{\ta}(\Mgb^k)=\Pic^{\ta}(\Mgb^k)_{\tf}\xleftarrow[\cong]{[P_{g,n}(k)]} \frac{\Gamma_{g,n}}{R_{g,n}},
\end{equation}
which concludes the proof. 
\end{proof}

\subsection{Tautological Picard group $\Pic^{\ta}(\Mg^k)$ and the restriction homomorphism}\label{S:item2}

The aim of this subsection is to prove Theorem \ref{T:Pic-taut}\eqref{T:Pic-taut2}. 

With this aim, it is convenient to introduce the following abelian groups.

\begin{defin}\label{D:abgrps2}
Fix an hyperbolic pair $(g,n)$ such that $g\geq 1$. 
\begin{enumerate}[(i)]
\item Let $\Gamma_{g,n}^o$ be the free abelian group generated by 
$$\un{\lambda}\: \text{ and } \{\un{\delta_{0,\{i\}}}, \un{\delta_{0,\{i\}^c}}\: : 1\leq i \leq n\}.$$
Let $R_{g,n}^o$ be the subgroup of $\Gamma_{g,n}^o$ generated by 
\begin{equation*}
\begin{sis}
& \un{\delta_{0,\{i\}}}-\un{\delta_{0,\{i\}^c}} & \text{ for any } 1\leq i \leq n, \\
& 10\un{\lambda} & \text{ if } g=2, \\
& 12\un{\lambda}& \text{ if } g=1, \\
& \un{\lambda}+\un{\delta_{0,\{p\}}}  \: \text{ for any } 1\leq p \leq n & \text{ if } g=1. \\
\end{sis}
\end{equation*}
Set $\Lambda_{g,n}^o:=\Gamma_{g,n}^o/R_{g,n}^o$. 

\item Let $\Gamma_{g,n}^{\partial}$ the free abelian group generated by 
$$ \un{\delta_{\irr}} \: \text{ and } \{\un{\delta_{a,A}}\: : 0\leq a \leq g, \emptyset \subseteq A\subseteq [n] \text{ with } |A|\geq 2 \: \text { if } a=0 \: \text{ and } \: |A|\leq g-2 \text{ if } a=g\}.$$

\end{enumerate}

\end{defin}

\begin{rmk}\label{R:abgrps}
From Definitions \ref{D:abgrps} and \ref{D:abgrps2}, it follows that there is an exact sequence 
$$0\to \Gamma_{g,n}^{\partial} \to \Gamma_{g,n}\xrightarrow{\rho} \Gamma_{g,n}^o\to 0, $$
 and that we have  
 $
 R_{g,n}^o=\rho(R_{g,n}).
 $
\end{rmk}

By the definition of $\Pic^{\ta}(\Mg^k)$, there is a surjective homomorphism 
\begin{equation}\label{E:Pgno}
\begin{aligned}
P_{g,n}^o: \Gamma_{g,n}^o & \twoheadrightarrow \Pic^{\ta}(\Mg^k), \\
\un{\lambda} & \mapsto \lambda, \\
\un{\delta_{0,\{i\}}}, \un{\delta_{g,\{i\}^c}} & \mapsto -\psi_i. \\
\end{aligned}
\end{equation}

\begin{proof}[Proof of Theorem \ref{T:Pic-taut}\eqref{T:Pic-taut2}]
We can assume that $g\geq 1$, since the case $g=0$ follows from \cite{Kee92}.
Consider the following commutative diagram 
\begin{equation}\label{E:diag-res}
\begin{tikzcd} 
\Gamma_{g,n}^{\partial} \ar[hook]{r}  \ar{d}{P_{g,n}^{\partial}} \arrow[bend right=60, swap]{dd}{\wt P_{g,n}^{\partial}}&   \Gamma_{g,n} \ar[twoheadrightarrow]{r}{\rho} \ar[twoheadrightarrow]{d}{P_{g,n}(k)} & \Gamma_{g,n}^o  \ar[twoheadrightarrow]{d}{P_{g,n}^o} \\
\ker(\res_{g,n}^{\ta}) \ar[hook]{r}  \ar[hook]{d} & \Pic^{\ta}(\Mgb^k)  \ar[twoheadrightarrow]{r}{\res_{g,n}^{\ta}} \ar[hook]{d} & \Pic^{\ta}(\Mg^k) \ar[hook]{d}\\
\ker(\res_{g,n}) \ar[hook]{r} & \Pic(\Mgb^k)  \ar[twoheadrightarrow]{r}{\res_{g,n}} & \Pic(\Mg^k)\\
\end{tikzcd}
\end{equation}
where:
\begin{itemize}
\item $\res_{g,n}$ is the restriction homomorphism and $\res_{g,n}^{\ta}:=(\res_{g,n})_{|\Pic^{\ta}(\Mgb^k)}$; 
\item the equality $P_{g,n}^o\circ \rho= \res_{g,n}^{\ta}\circ P_{g,n}(k)$ follows from the definitions \eqref{E:Pgn} and \eqref{E:Pgno}; 
\item $P_{g,n}^{\partial}$ is the restriction of $P_{g,n}(k)$ to $\Gamma_{g,n}^{\partial}$ and $\wt P_{g,n}^{\partial}$ is the composition of $P_{g,n}^{\partial}$ with the inclusion $\ker(\res_{g,n}^{\ta})\hookrightarrow \ker(\res_{g,n})$;
\item the first row is exact by Remark \ref{R:abgrps};
\item the middle row is exact because $\res_{g,n}^{\ta}$ is surjective by the definition of the tautological subgroups;
\item the last row is is exact because $\res_{g,n}$ is surjective since $\Mgb^k$ is smooth over $k$ (see e.g.  \cite[Prop. 1.9]{PTT15}).
\end{itemize}

By construction, the image of the composition 
$$\Gamma_{g,n}^{\partial} \xrightarrow{\wt P_{g,n}^{\partial}} \ker(\res_{g,n})\hookrightarrow \Pic(\Mgb^k)$$ 
is the subgroup of $\Pic(\Mgb^k)$ generated by the boundary divisors of $\Mgb^k$ and this subgroup is exactly $\ker(\res_{g,n})$ by  \cite[Prop. 1.9]{PTT15}. Hence we deduce that:
\begin{enumerate}[(a)]
\item $\ker(\res_{g,n}^{\ta})=\ker(\res_{g,n})$;
\item $P_{g,n}^{\partial}$ is surjective.  
\end{enumerate}  
From (a), we get that $\res_{g,n}$ induces the isomorphism \eqref{E:res-iso}. From (b) together with the fact that $\ker(P_{g,n}(k))=R_{g,n}$ (see \eqref{E:pres-Pict}), we get that $\ker(P_{g,n}^o)=\rho(R_{g,n})$, which coincides with $R_{g,n}^o$ by Remark \ref{R:abgrps}.  Hence the morphism $P_{g,n}^o$ induces the isomorphism
$$
\frac{\Gamma_{g,n}^o}{R_{g,n}^o}\xrightarrow[{[P_{g,n}^o]}]{\cong} \Pic^{\ta}(\Mg^k),
$$
from which we deduce the explicit presentation \eqref{E:ex-Pic-op}. 
\end{proof}

\subsection{Weak Franchetta Conjecture and the forgetful homomorphism}\label{S:item3}
The aim of this subsection is to prove Theorem \ref{T:Pic-taut}\eqref{T:Pic-taut3}.

The proof is based on the   ``Weak Franchetta Conjecture", which  is the computation of the relative Picard of the universal curve $\pi: \mt C^k_{g,n}\to \mt M^k_{g,n}$:
$$
\RPic(\Cg^k):=\Pic(\Cg^k)/\pi^*\Pic(\Mg^k).
$$

Note that when $\Mg^k$ is generically a scheme (which happens precisely when $g+n\geq 3$),  the restriction to the generic fiber $\cC_{\eta_{g,n}}\to \eta_{g,n}=\Spec  k(\Mg^k)$ of $\pi: \Cg^k\to \Mg^k$  induces an isomorphism 
(see \cite[Prop. 2.3.2]{FV22})
\begin{equation}\label{E:Picgencur}
 \RPic(\Cg^k)\xrightarrow{\cong} \Pic(\cC_{\eta_{g,n}}).
\end{equation}

\begin{teo}[Weak Franchetta Conjecture]\label{franchetta}
Let $g,n\geq 0$ such that $2g-2+n>0$. 
 The group $\RPic(\mathcal C_{g,n}^k)$ is generated by the relative dualizing line bundle $\omega_{\pi}$ and the line bundles 
$\{\oo(\sigma_1), \ldots, \oo(\sigma_n)\}$ associated to the universal sections $\sigma_1,\ldots,\sigma_n$, subject to the following relations:
\begin{itemize}
\item if $g=1$ then $\omega_{\pi}=0$;
\item if $g=0$ then $\oo(\sigma_1)=\ldots= \oo(\sigma_n)$ and $\omega_{\pi}=\oo(-2\sigma_1)$.
\end{itemize}
\end{teo}
The above result was proved for $g\geq 3$ by Arbarello-Cornalba \cite{AC87} in characteristic zero and by Schr\"oer \cite{Sch03} in arbitrary characteristic. 
Moreover, the result for $\RPic(\mathcal C^k_{g,n})_{\bbQ}$ follows   from the computation of $\Pic(\ov \cM_{g,n}^k)_{\bbQ}$ performed by Arbarello-Cornalba \cite{AC98} in characteristic zero and by Moriwaki \cite{Mor01} in positive characteristic. 


\begin{proof}
We will distinguish two cases, according to whether or not $\Mg^k$ is generically a scheme. 

\fbox{Case I: $g+n\geq 3$}

Under this assumption, using the isomorphism \eqref{E:Picgencur}, we have to show  that 
 $\Pic(\cC_{\eta_{g,n}})$ is generated by the dualizing line bundle $\omega_{\cC_{\eta_{g,n}}}$ and the line bundles $\{\oo(p_1),\ldots, \oo(p_n)\}$ associated to the marked points 
 $\{p_1:=\sigma_1(\eta_{g,n}), \ldots, p_n:=\sigma_n(\eta_{g,n})\}$, subject to the following relations:
\begin{itemize}
\item if $g=1$ then $\omega_{\cC_{\eta_{g,n}}}=0$;
\item if $g=0$ then $\oo(p_1)=\ldots= \oo(p_n)$ and $\omega_{\cC_{\eta_{g,n}}}=\oo(-2p_1)$.
\end{itemize}

The case  $g=0$ is easy. Indeed, the generic curve $\cC_{\eta_{0,n}}$ is a geometrically connected smooth projective curve of genus $0$ with at least one $K:=k(\cM_{0,n}^k)$-rational point (since $n\geq 1$), and hence $\cC_{\eta_{0,n}}$ is isomorphic (over $K$) to $\bbP^1_{K}$. Now the result follows from the fact that $\Pic( \bbP^1_{K})$ is infinite cyclic generated by $\oo(1)$ and the fact that $\oo(p_i)=\oo(1)$ and $\omega_{\cC_{\eta_{0,n}}}=\oo(-2)$.  

The case  $g\geq 1$ is essentially due to Schr\"oer \cite{Sch03}, although he assumes the stronger hypothesis that $g\geq 3$ (probably in order to avoid the forbidden case $(g,n)=(2,0)$). Let us briefly review the essential ingredients in his proof in order to convince the reader that the  proof of loc. cit. works under the weaker hypothesis that $g+n\geq 3$. 
Consider the free abelian group  $P:=\bbZ \oo(p_1)\oplus \ldots\oplus \bbZ \oo(p_n)\oplus \bbZ \omega_{\cC_{\eta_{g,n}}}$ and the tautological morphism 
\begin{equation}\label{E:phi-tau}
\phi: P:=\bbZ \oo(p_1)\oplus \ldots\oplus \bbZ \oo(p_n)\oplus \bbZ \omega_{\cC_{\eta_{g,n}}}\to \Pic(\cC_{\eta_{g,n}}).  
\end{equation}
Now the proof is divided in three steps.

\un{Step I:} The kernel of $\phi$ is equal to  
$$\ker(\phi)=
\begin{cases} 
\bbZ \omega_{\cC_{\eta_{g,n}}} & \text{ if } g=1, \\
\{0\} & \text{ if } g\geq 2.
\end{cases}$$

This follows from the existence of a stable $n$-marked genus-$g$ curve $(X, x_1,\ldots, x_n)$  over a field $F\supset k$ whose normalization $\nu:\wt X\to X$  is an elliptic curve with the property that the line bundles $\{\nu^*\omega_X, \nu^*\oo(x_1),\ldots, \nu^*\oo(x_n)\}$ if $g\geq 2$ (resp.  the line bundles $\{\nu^*\oo(x_1),\ldots, \nu^*\oo(x_n)\}$ if $g=1$) are linearly independent in $\Pic(\wt X)$; see the construction in \cite[p. 379]{Sch03}.

\un{Step II:} The cokernel of $\phi$ is torsion-free. 

This is proved in \cite[Prop. 4.1]{Sch03} using a degeneration argument  towards the curve $X$ mentioned in Step I together with the fact that $\Pic^o_{\cC_{\eta_{g,n}}}(\eta_{g,n})$ is torsion-free. This last property is proved  in \cite[Prop. 3.1]{Sch03} by a degeneration argument towards a stable $n$-marked genus-$g$ curve $(Y, y_1,\ldots, y_n)$ of compact type over  $k(T)$ with the property that $\Pic^o_Y$ is an ordinary abelian variety such that $\Pic^o_Y(k(T))=0$. 
 The curve $(Y,y_1,\ldots, y_n)$ is constructed in \cite[p. 377]{Sch03}  by gluing nodally to $\bbP^1_{k(T)}$, with $n$-marked points, $g$ copies of an ordinary elliptic curve $E$ over  $k(T)$ such that $\Pic^o_E(k(T))=0$ (whose existence is showed in \cite[Prop. 3.2]{Sch03} using special Halphen pencils). Note that the curve $(Y,p_1,\ldots, p_n)$ is stable if and only if $g+n\geq 3$  (the stronger condition $g\geq 3$ mentioned  in \cite[p. 377]{Sch03} is not needed!).

\un{Step III:} The cokernel of $\phi$ is torsion.

Indeed, take a line bundle $L\in \Pic(\cC_{\eta_{g,n}})$. Because of the isomorphism \eqref{E:Picgencur}, we can find a line bundle $\cL$ on $\Cg^k$ such that $\cL_{|\cC_{\eta_{g,n}}}=L$. Using that  $\Pic(\ov \cM_{g,n}^k)_{\bbQ}$ is generated by the tautological line bundles (in arbitrary characteristic) as proved in  \cite{AC98} and \cite{Mor01}, it follows that there exists an integer $m$ such that 
$$m\cL=a\lambda+\sum_{i=1}^n b_i \psi_i + \sum_{i=1}^n c_i \oo(\sigma_i) +d \omega_{\pi}.$$ 
By restricting to the generic curve $\cC_{\eta_{g,n}}$, we get 
$$mL=\sum_{i=1}^n (b_i+c_i)\oo(p_i) +d\omega_{\cC_{\eta_{g,n}}}, $$
which shows that $mL\in \Im(\phi)$ and it concludes the proof of Step III. 

\fbox{Case II: $(g,n)=(1,1)$ and $(2,0)$}

In this case, the stack $\Mg^k$ is a DM stack with a non-trivial generic residue gerbe (hence it is not generically a scheme). However, a small modification of the argument in Case I still works. First of all, we can replace the stack with a scheme. More precisely, there exists a cartesian diagram of stacks
$$
\begin{tikzcd} 
\overline C\ar{r}{f}\ar{d}{\mathrm{ap}}  &\overline{\cC}_{g,n}^k \ar{d}{\overline{\pi}}&\Cg^k\ar{d}{\pi}\ar[hook']{l} \\
\overline M\ar{r}&\Mgb^k& \Mg^k \ar[hook']{l}
\end{tikzcd}
$$
such that:
\begin{enumerate}[(i)]
\item $\overline{\pi}:\overline{\cC}_{g,n}^k\to \Mgb^k$ is the universal curve,
\item $\overline M$ is quasi-projective smooth $k$-variety; 
\item there is an isomorphism  $\RPic(\Cg^k)\xrightarrow{\cong} \Pic(C_\eta)$, where the isomorphism is induced by the restriction to the generic fibre $C_\eta$ of the family $\mathrm{ap}$;  
\item the morphism $\Hom_k(\Spec(K),\ov M)\twoheadrightarrow\Hom_k(\Spec(K),\Mgb^k)$ of sets is surjective, for any field $K\supseteq k$.
\end{enumerate}
The first two points follows by \cite{EG98}, since $\Mgb^k$  is a quotient stack (the variety $\ov M$ is called in loc. cit. an equivariant approximation of $\Mgb^k$). The third one by \cite[Def.-Prop. 2.2(3)]{FP16}, because $K$ is an extension of the algebraically closed base field $k$, hence infinite. The strategy of \cite{Sch03}, i.e. Case I, still applies if we replace the universal curve $\overline{\pi}:\overline{\cC}_{g,n}^k\to \Mgb^k$ with the family $\mathrm{ap}:\ov C\to\ov M$. The property (iv) guarantees that all the curves in \emph{loc. cit.} can be lifted on the family $\mathrm{ap}$. The other changes needed are the choice of the stable curves in Step I and II. More precisely: we define the tautological morphism 
\begin{equation}\label{E:phi-tau2}
\phi: P:=\bbZ \oo(p_1)\oplus \ldots\oplus \bbZ \oo(p_n)\oplus \bbZ \omega_{C_{\eta}}\to \Pic(C_{\eta})
\end{equation}
as in \eqref{E:phi-tau}, and we need to prove that it is an isomorphism.

\un{Step I:} The kernel of $\phi$ is equal to  
$$\ker(\phi)=
\begin{cases} 
\bbZ \omega_{C_\eta} & \text{ if } (g,n)=(1,1), \\
\{0\} & \text{ if } (g,n)=(2,0).
\end{cases}$$
If $(g,n)=(1,1)$, the argument of Step I in the previous case still applies if we use the curve $(X,x_1,\ldots,x_n):=(E,p)$, where $E$ is the elliptic curve of \cite[Prop. 3.2]{Sch03} with the unique rational point $p$. If $(g,n)=(2,0)$, we need to use the curve $(X,x_1,\ldots,x_n):=(E\cup_p E)$, obtained by attaching two copies of $E$ nodally along the  point $p$.

\un{Step II:} The cokernel of $\phi$ is torsion-free.

If $\eta\in \ov M$ is the generic point, the group $\Pic^o_{\ov C}(\eta)$ is torsion-free. Indeed, the proof of \cite[Prop. 3.1]{Sch03} still works if we use the smooth marked curve $(E,p)$, when $(g,n)=(1,1)$, and the stable curve $E\cup_p E$, when $(g,n)=(2,0)$. Then, the cokernel of $\phi$ is torsion-free by repeating the degeneration argument of \cite[Prop. 4.1]{Sch03} applied at the same curves as above.

\un{Step III:} The cokernel of $\phi$ is torsion. 

It goes through with no changes.

\end{proof}

\begin{rmk}
Let us point out that the Weak Franchetta conjecture holds true also for the pairs $(g,n)$ such that $2g-2+n\leq 0$ (and this case was needed in \cite{FV22}). 
In particular, $\RPic(\mathcal C_{0,0}^k)$ is freely generated by $\omega_{\pi}$.
Indeed:

\begin{itemize}

\item If $(g,n)=(1,0)$ then it follows from \cite[Thm. 1.1]{FO} that the Picard group of $\cM_{1,1}^k=\cC_{1,0}^k$ is a cyclic group of order $12$ generated by the Hodge line bundle $\lambda:=\pi_*(\omega_{\pi})$. 
Since the Hodge line bundle $\lambda$ on $\cM_{1,1}^k$  is the pull-back of the Hodge line bundle on $\cM_{1,0}^k$, we conclude that $\RPic(\mathcal C_{1,0}^k)=0$.


\item If $g=0$ and $n=0,1,2$ then the universal family $\pi:\Cg^k\to \Mg^k$ admits the following presentation 
$$
 [\bbP^1/G] \to \cB G \quad \text{ with }
 G=\begin{cases}
\PGL_2  & \text{ if } g=0 \:  \text{ and }\: n=0, \\
B\cong  \Ga\rtimes \Gm & \text{ if } g=0 \:  \text{ and }\: n=1, \\ 
T\cong \Gm & \text{ if } g=0 \:  \text{ and }\: n=2. \\ 
\end{cases}
$$
where $B$ is the Borel subgroup of $\PGL_2$ consisting of upper triangular matrices,  $T$ is the maximal torus of $\PGL_2$ consisting of diagonal matrices and the action of $\PGL_2$ on $\bbP^1$ is the natural one. 


From the above explicit presentation of the universal family, it follows that $\Pic(\cM_{0,n}^k)$ is isomorphic to the group $\Lambda^*(G):=\Hom(G,\Gm)$ of characters  of $G$ while $\Pic(\cC_{0,n}^k)$ is isomorphic to the group $\Pic^G(\bbP^1)$ of $G$-linearized line bundles on $\bbP^1$. Moreover, there is a natural exact sequence 
$$
0\to \Lambda^*(G)\xrightarrow{\alpha} \Pic^G(\bbP^1) \xrightarrow{\beta} \Pic(\bbP^1), 
$$
where the map $\alpha$ can be identified with the pull-back map $\pi^*:\Pic(\cM_{0,n}^k)\to \Pic(\cC_{0,n}^k)$ and the image of $\beta$ consists of all the line bundles of $\bbP^1$ that admit a $G$-linearization.  We conclude by using the well-known fact that $\omega_{\bbP^1}=\cO(-2)$ admits a $\PGL_2$-linearization while 
$\cO(1)$ admits a $B$-linearization (and hence also a $T$-linearization) but not a $\PGL_2$-linearization.

\end{itemize}
\end{rmk}

Using the Weak Franchetta Conjecture, we can now give a proof of Theorem \ref{T:Pic-taut}\eqref{T:Pic-taut3}.

\begin{proof}[Proof of Theorem  \ref{T:Pic-taut}\eqref{T:Pic-taut3}]
For any hyperbolic pair $(g,m)$, the stack $\cM^k_{g,m+1}$ is an open substack of $\cC^k_{g,m}$ whose complement is equal to 
$$
\cC^k_{g.m}\setminus \bigcup_{i=1}^m \Im(\sigma_i),
$$
where $\sigma_i$ are the  sections of the universal family $\pi:\cC_{g.m}^k\to \cM^k_{g.m}$. 
Since the stack $\cC^k_{g,m}$ is smooth, we have an exact sequence (see \cite[Prop. 1.9]{PTT15}) 
$$0 \to \Sigma_{g,m}:=\langle \cO(\sigma_1), \ldots, \cO(\sigma_m)\rangle \to \Pic(\cC^k_{g,m})\to \Pic(\cM^k_{g,m+1})\to 0.$$
Consider now the morphism $F_m:\cM^k_{g,m+1}\to \cM^k_{g,m}$ that forgets the last marked point. 
By applying the snake lemma to the commutative diagram with exact rows 
$$\begin{tikzcd}
0 \ar[r] & 0 \ar[r] \ar[d]  & \Pic(\cM^k_{g,m}) \ar[equal]{r} \arrow[hook]{d}& \Pic(\cM^k_{g,m}) \ar{r} \ar{d}{F_m^*}& 0 \\
0 \ar[r] & \Sigma_{g,m} \ar[r] &  \Pic(\cC^k_{g,m})\ar[r] & \Pic(\cM^k_{g,m+1})\ar[r] & 0
\end{tikzcd}$$ 
we obtain the following exact sequence 
$$
0\to \ker(F_m^*)\to \Sigma_{g,m} \to \RPic(\cC^k_{g,m})\to \coker(F_m^*)\to 0.
$$
Theorem \ref{franchetta} implies (using the hypothesis that $g\geq 1$) that the line bundles $\{\cO(\sigma_1),\ldots, \cO(\sigma_m)\}$ are linearly independent in $\RPic(\cC^k_{g,m})$; hence the map $\Sigma_{g,m}\to \RPic(\cC^k_{g,m})$ is injective, and we get
\begin{equation}\label{E:kerF}
\ker(F_m^*)=\{0\}.
\end{equation}
Moreover, again from Theorem \ref{franchetta} with $g\geq 1$, we get that  the cokernel of the map $\Sigma_{g,m}\to \RPic(\cC^k_{g,m})$ is zero if $g=1$ and it is free of rank one generated by the restriction of the relative dualizing sheaf $\omega_{\cC^k_{g,m}/\cM^k_{g,m}}$ if $g\geq 2$. Using that the restriction of $\omega_{\cC^k_{g,m}/\cM^k_{g,m}}$ to $\cM^k_{g,m+1}$ is equal to $\psi_{m+1}$ (see \cite[Thm. 4.1(d)]{Knu3}), we obtain that 
\begin{equation}\label{E:cokerF}
\coker(F_m^*)=
\begin{cases} 
\{0\} & \text{ if } g=1,\\
\bbZ\langle \psi_{m+1} \rangle & \text{ if } g\geq 2. 
\end{cases}
\end{equation}
Using that $F=F_{n-1}\circ \ldots \circ F_0:\Mg^k\to \cM_g^k$ if $g\geq 2$ (resp. $F=F_{n-1}\circ \ldots \circ F_1:\cM^k_{1,n}\to \cM^k_{1,1}$ if $g\geq 1$) and putting together \eqref{E:kerF} and \eqref{E:cokerF}, we get that the pull-back $F^*$ on the Picard groups is injective and that 
\begin{equation}\label{E:cokerpull}
\begin{sis}
& \frac{\Pic(\Mg^k)}{F^*\Pic(\cM_g^k)}=\bbZ\langle \psi_1\rangle \oplus \ldots  \oplus \bbZ\langle \psi_n\rangle & \text{ if } g\geq 2, \\
& \frac{\Pic(\cM^k_{1,n})}{F^*\Pic(\cM^k_{1,1})}=\{0\} & \text{ if } g=1.
\end{sis}
\end{equation}
Comparing \eqref{E:cokerpull} with Theorem \ref{T:Pic-taut}\eqref{T:Pic-taut2}, we get that 
\begin{equation}\label{E:cokeriso}
\begin{sis}
& \frac{\Pic^{\ta}(\Mg^k)}{F^*\Pic^{\ta}(\cM^k_g)}\xrightarrow{\cong} \frac{\Pic(\Mg^k)}{F^*\Pic(\cM^k_g)}& \text{ if } g\geq 2, \\
&  \frac{\Pic^{\ta}(\cM^k_{1,n})}{F^*\Pic^{\ta}(\cM^k_{1,1})}= \frac{\Pic(\cM^k_{1,n})}{F^*\Pic(\cM^k_{1,1})}=\{0\} & \text{ if } g=1.
\end{sis}
\end{equation}
By the snake lemma, the above isomorphisms imply that \eqref{E:pull-iso} holds true, and we are done.

\end{proof}

\subsection{Torsion subgroup of $\Pic(\Mgb^k)$}\label{S:item4}

The aim of this subsection is to prove Theorem \ref{T:Pic-taut}\eqref{T:Pic-taut4}. Recall that we are assuming that $k=\ov k$ is algebraically closed, throughout this section.


A crucial ingredient in the proof is that $\Mgb^k$ is \emph{algebraically simply connected}, i.e. any  connected \emph{cover} of $\Mgb^k$ (i.e.  finite and \'etale representable morphism $f:\cX\to \Mgb^k$ with $\cX$ connected)  is trivial (i.e. $f:\cX\to\Mgb^k$ is an isomorphism)\footnote{this is equivalent to the fact that the \'etale fundamental group $\pi_1^{et}(\Mgb^k,\ov x)$ (with respect to some  geometric point $\ov x$) vanishes; see \cite{Noo04} for the definition and main properties of the \'etale fundamental group of an algebraic stack.}.


\begin{teo}\label{T:simply-con}
The stack $\Mgb^k$ is algebraically simply connected.
\end{teo}
\begin{proof}
We will divide the proof in three cases, according to the field $k$. 
\begin{enumerate}[(a)]
\item If $k=\bbC$ then, by the Riemann existence theorem for stacks (see \cite[Thm. 20.4]{Noo}), the functor $$
\begin{aligned}
\left\{\text{Covers of } \Mgb^{\bbC} \right\} & \longrightarrow \left\{\text{Finite (topological) covers of } \Mgb^{\an}  \right\} \\
\cX\to \Mgb^{\bbC} & \mapsto \cX^{\an} \to \Mgb^{\an}
\end{aligned}
$$
is an equivalence of categories. Hence the result follows from the fact that $\Mgb^{\an}$ is (topologically) simply connected by Proposition \ref{P:simplycon}.


\item If $k(=\ov k)$ has characteristic zero, then the result follows from the previous case $k=\bbC$ together with the fact that, for any extensions $k\subseteq k'$ of algebraically closed fields,  the  functor 
$$
\begin{aligned}
\left\{\text{Covers of } \Mgb^k \right\} & \longrightarrow \left\{\text{Covers of } \Mgb^{k'}  \right\} \\
\cX\to \Mgb^k & \mapsto \cX_{k'}:=\cX\times_k k' \to \Mgb^{k'}
\end{aligned}
$$
is an equivalence of categories, a fact that can be proved as in \cite[\href{https://stacks.math.columbia.edu/tag/0A49}{Tag 0A49}]{stacks-project}. 
\item If $k(=\ov k)$ has characteristic $p>0$, then consider the morphism $\Mgb^{W(k)}\to \Spec W(k)$, where $W(k)$ is the ring of Witt vectors over $k$. The conclusion will follow from 
the following two statements.

\begin{itemize}
\item Any (resp. connected) cover of $\Mgb^k$ lifts to a (resp. connected) cover of $\Mgb^{W(k)}$.

The proof of this statement is similar to the one contained in \cite[\href{https://stacks.math.columbia.edu/tag/0A48}{Tag 0A48}]{stacks-project} and it proceeds like this. 
Let $f_1:\cX_1\to \Mgb^k$ be a cover. For any $n\geq 1$, consider the local Artin ring $W_n(k)=W(k)/\mathfrak{m}^n$ where $\mathfrak{m}$ is the maximal ideal of $W(k)$. We get a chain of closed embeddings 
$$
\Mgb^{W_1(k)=k}\hookrightarrow \Mgb^{W_2(k)}\hookrightarrow \Mgb^{W_3(k)}\hookrightarrow \ldots...
$$
where each closed embedding  is defined by a square-zero ideal sheaf. By \cite[\href{https://stacks.math.columbia.edu/tag/039R}{Tag 039R}]{stacks-project} (which can be used since $f_1$ is representable), we can lift $f_1$ to a tower of covers
$$
f_n:\cX_n\to \Mgb^{W_n(k)}
$$
for any $n\geq 1$, such that the restriction of $f_n$ over $\Spec W_{n-1}(k)$ coincides with $f_{n-1}$. Since $W(k)=\varprojlim W_n(k)$, by the Grothendieck's existence existence theorem for algebraic stacks (see \cite[Thm. 1.4]{Ols05}), there exists a finite representable morphism $f:\cX\to \Mgb^{W(k)}$ whose restriction over $\Mgb^{W_n(k)}$ coincides with $f_n$ for any $n\geq 1$. Arguing as in \cite[\href{https://stacks.math.columbia.edu/tag/0A43}{Tag 0A43}]{stacks-project}, 
it follows that $f$ is formally smooth, and hence \'etale. Therefore $f:\cX\to \Mgb^{W(k)}$ is a cover that lifts the initial cover $f_1:\cX_1\to \Mgb^k$.  Clearly, if $f_1$ is connected then also $f$ must be connected. 

\item Any connected cover of $\Mgb^{W(k)}$ is trivial.

Let $f:\cX\to \Mgb^{W(k)}$ be a connected cover. Consider the Stein factorization (see \cite[Thm. 11.3]{Ols07}) of the proper and  surjective morphism $F:\cX\xrightarrow{f} \Mgb^{W(k)}\to S:=\Spec W(k)$:
$$
F:\cX\xrightarrow{g} \un{\Spec}_S(F_*\cO_{\cX}) \xrightarrow{h} S=\Spec W(k).
$$
By \cite[Thm. 11.3]{Ols07}, the morphism $g$ is a surjective and proper morphism with geometrically connected fibers.

On the other hand, the morphism $h$ is a representable (surjective) finite morphism since $F_*\cO_{\cX}$ is a coherent sheaf on $S$ by loc. cit.
Since the morphism $F$ is smooth (hence flat with geometrically reduced fibers), arguing as in \cite[\href{https://stacks.math.columbia.edu/tag/0BUN}{Tag 0BUN}]{stacks-project} we deduce that the morphism $h$ is also \'etale. Furthermore,  $\un{\Spec}_S(F_*\cO_{\cX})$ is connected being the image via the surjective morphism $g$ of the connected stack $\cX$. In other words, the morphism $h$ is a connected cover of $S=\Spec W(k)$. Since $W(k)$ is a strictly Henselian local ring (being a complete local ring with algebraically closed residue field), we deduce that $h$ must be an isomorphism (see e.g. \cite[\href{https://stacks.math.columbia.edu/tag/04GK}{Tag 04GK}]{stacks-project}). We conclude that $F=g$, which implies that $F:\cX\to \Spec W(k)$ has geometrically connected fibers. 

Now, restricting $f$ to the geometric generic fiber $\ov \eta=\Spec \ov{F(W(k))}$, where $\ov{F(W(k))}$ is an algebraic closure of the fraction field $F(W(k))$ of $W(k)$, we obtain a cover $f_{\ov\eta}:\cX_{\ov \eta}\to \Mgb^{F(W(k))}$ which is furthermore connected since $F$ has geometrically connected fibers. Since $\char \ov{F(W(k))}=0$, by what proved above $f_{\ov\eta}$ must be an isomorphism. This implies that $f$ has relative degree $1$, hence it must be an isomorphism.
\end{itemize}

\end{enumerate}
\end{proof}

\begin{rmk}\label{R:specia}
Observe that part (c) of the above proof would  follow from the existence and the surjectivity of a specialization map among the \'etale fundamental groups (as defined in  \cite{Noo04}) of the geometric fibers of a proper flat morphisms of stacks $f:\cX\to S$ with geometrically connected and reduced fibers, extending the well-know case of morphisms of schemes (see e.g. \cite[\href{https://stacks.math.columbia.edu/tag/0C0P}{Tag 0C0P}]{stacks-project}). However, we are not aware of such an extension to morphisms of stacks, not even for DM stacks.  
\end{rmk}

Using the above Theorem \ref{T:simply-con}, we can now study  the $k$-group scheme $\PIC^{\tau}_{\Mgb^k/k}$ and the torsion subgroup $\Pic(\Mgb^k)_{\tor}\subset \Pic(\Mgb^k)$.

\begin{prop}\label{P:tau-tors}
\noindent 
\begin{enumerate}[(i)]
\item \label{P:tau-tors1} The Picard group $\Pic(\Mgb^k)$ coincides with $\NS(\Mgb^k)$ and it is finitely generated.
\item \label{P:tau-tors2} If $\char(k)=0$ then $\PIC^{\tau}_{\Mgb^k/k}$ is the trivial $k$-group scheme and $\Pic(\Mgb^k)$ is a torsion-free group.
\item \label{P:tau-tors3} If $\char(k)=p>0$ then $\PIC^{\tau}_{\Mgb^k/k}$ is a finite $k$-group scheme and we have that 
$$
\Pic(\Mgb^k)_{\tor}=\PIC^{\tau}_{\Mgb^k/k}(k)
$$
is a finite $p$-group.
\end{enumerate}
\end{prop}
\begin{proof}
Let us first show the following

\un{Claim}: If $L$ is an $s$-torsion line bundle on $\Mgb^k$ such that $\char(k)\not\vert s$ then $L$ is trivial.  

Indeed, denote by $1\leq r\vert s$ the order of $L$ in  $\Pic(\Mgb^k)$. 
The line bundle $L$ gives rise to a $\mu_r$-cyclic connected finite and flat representable morphism 
$$\Mgb^k[\sqrt[r]{L}]:=\Spec_{\cO_{\Mgb^k}}\left(L^0\oplus \ldots \oplus L^{r-1} \right) \to \Mgb^k.$$
 Since $\mu_r$ is \'etale over $k$ by our assumption that $\char(k)$ does not divide $r$, the morphism  $\Mgb^k[\sqrt[r]{L}]\to \Mgb^k$ is also \'etale, hence it is a connected cover of $\Mgb^k$. Since $\Mgb^k$ is algebraically simply connected by Theorem \ref{T:simply-con},  the connected cover $\Mgb^k[\sqrt[r]{L}]\to \Mgb^k$ must be trivial, which forces $r=1$. 
  
\vspace{0.1cm}

Consider the connected component $\PIC^o_{\Mgb^k/k}$ containing the origin of the finite type $k$-group scheme $\PIC^{\tau}_{\Mgb^k/k}$.
Proposition \ref{P:Pico-tau} implies (using that $\Mgb^k$ is $k$-smooth) that $\PIC^o_{\Mgb^k/k}$ is a (connected) projective $k$-group scheme, and hence  that the reduced underlying scheme $(\PIC^o_{\Mgb^k/k})_{\red}$ is an abelian variety over $k$. 
By the above Claim, $(\PIC^o_{\Mgb^k/k})_{\red}(k)$ does not contain non-trivial torsion elements of order coprime with $\char(k)$, and hence it must be trivial. This implies that 
$$\Pic^o(\Mgb^k)=\PIC^o_{\Mgb^k/k}(k)=0,$$ 
from which it follows that 
$$\NS(\Mgb^k)=\Pic(\Mgb^k).$$
Proposition \ref{P:finit-NS} implies therefore that $\Pic(\Mgb^k)$ is finitely generated (which concludes the proof of \eqref{P:tau-tors1}) and that its (finite) torsion subgroup is equal to 
$$\Pic(\Mgb^k)_{\tor}=\Pic^{\tau}(\Mgb^k)=\PIC^{\tau}_{\Mgb^k/k}(k).$$
Moreover, the above Claim implies that $\Pic(\Mgb^k)_{\tor}$ is trivial if $\char(k)=0$ and it is a $p$-group if $\char(k)=p>0$. 

We now conclude the proofs of \eqref{P:tau-tors2} and \eqref{P:tau-tors3} observing that 
\begin{itemize}
\item $\PIC^{\tau}_{\Mgb^k/k}$ is a finite $k$-group scheme, being of finite type over $k$ with a finite connected component $\PIC^{o}_{\Mgb^k/k}$ containing the identity element;
\item $\PIC^{\tau}_{\Mgb^k/k}$ is $k$-\'etale  if $\char(k)=0$, being of $k$-finite and $k$-smooth by Proposition \ref{P:smPic}\eqref{P:smPic2}.  
\end{itemize}
\end{proof}

We  can now give a  proof of Theorem \ref{T:Pic-taut}\eqref{T:Pic-taut4}. 

\begin{proof}[Proof of Theorem \ref{T:Pic-taut}\eqref{T:Pic-taut4}]
From Proposition \ref{P:tau-tors}, it follows that $\Pic(\Mgb^k)_{\tor}$ is trivial if $\char(k)=0$ and it is a finite $p$-group if $\char(k)=p>0$. 

Using parts \eqref{T:Pic-taut1}, \eqref{T:Pic-taut2} and \eqref{T:Pic-taut3} of Theorem \ref{T:Pic-taut}, we get that $\Pic^{\ta}(\Mgb^k)=\Pic(\Mgb^k)$ (which implies that $\Pic(\Mgb^k)_{\tor}=0$ since 
$\Pic^{\ta}(\Mgb^k)$ is torsion-free by part \eqref{T:Pic-taut1}) in each of the following cases:
\begin{itemize}
\item when  $g=0$  (and $n\geq 3$) by \cite{Kee92};
\item when $g=1$ (and $n\geq 1)$  since $\Pic^{\ta}(\cM_{1,1}^k)=\Pic(\cM_{1,1}^k)$ by \cite{FO};
\item when $g=2$  since $\Pic^{\ta}(\cM_{2}^k)=\Pic(\cM_{2}^k)$ by \cite{Vis98} if  $\char(k)\neq 2$\footnote{The hypothesis that $\char(k)$ is also different from $3$ in loc. cit. is only need in the computation of the higher Chow groups of $\cM_2$.}
 and \cite{Ber} if $\char(k)=2$;
\item when $g=3,4,5$  since, for these values of $g$, we have that  $\Pic^{\ta}(\cM_{g}^k)=\Pic(\cM_{g}^k)$ by \cite{diL}.
\end{itemize}

It remains to prove the decomposition \eqref{E:splitPic} (for which we can assume that $g\geq 6$ for otherwise it holds trivially for what said above) and the isomorphism \eqref{E:tor-pull}.
We will divide the proof in three steps.

\un{Step I:} \: \eqref{E:splitPic} holds true for $n=0$.

As observed above, we can (and will) assume that $g\geq 3$. We first claim that the composition 
\begin{equation}\label{E:comp-mor}
\Pic^{\ta}(\ov \cM_g^k)\hookrightarrow \Pic(\ov\cM_g^k)\twoheadrightarrow  \Pic(\ov\cM_g^k)_{\tf}
\end{equation}
is an isomorphism. 

Indeed, since $\Pic^{\ta}(\ov\cM_g^k)$ is torsion-free by part \eqref{T:Pic-taut1}, the  homomorphism \eqref{E:comp-mor} is injective. Furthermore, using that $ \Pic^{\ta}(\ov\cM_g^k)_{\bbQ}=\Pic(\ov\cM_g^k)_{\bbQ}$ by \cite[Thm. 2.2]{AC98} in characteristic zero and \cite{Mor01} in positive characteristic, we deduce that the injective homomorphism \eqref{E:comp-mor} has finite index. 
Hence, in order to prove that \eqref{E:comp-mor} is an isomorphism, it is enough to construct $\dim_{\bbQ}(\Pic(\ov\cM_g^k)_{\bbQ})$ (integral $k$-proper) curves in $\ov\cM_g^k$ such that the intersection matrix of a basis of $\Pic^{\ta}(\ov\cM_g^k)$ with these test curves is invertible. This is proved  in \cite[\S 2]{AC87}, under the assumption that $g\geq 3$.

Now we can finish the proof of Step I. Indeed, the isomorphism in \eqref{E:comp-mor} provides a canonical splitting of the exact sequence 
$$0\to \Pic(\ov\cM_g^k)_{\tor}\to \Pic(\ov\cM_g^k)\to \Pic(\ov\cM_g^k)_{\tf}\to 0,$$
by identifying the torsion-free quotient $\Pic(\ov\cM_g^k)_{\tf}$ with the tautological subgroup $\Pic^{\ta}(\ov\cM_g^k)$. Hence we get the decomposition \eqref{E:splitPic} for $n=0$.

\un{Step II:}\:  If $g\geq 2$ then the pull-back via the forgetful morphism $\ov F$ induces an isomorphism
\begin{equation}\label{E:defect-ta}
[\ov F^*]:\frac{\Pic(\ov\cM_g^k)}{\Pic^{\ta}(\ov\cM_g^k)}\xrightarrow{\cong} \frac{\Pic(\Mgb^k)}{\Pic^{\ta}(\Mgb^k)}.
\end{equation}

Indeed, consider the commutative diagram 
$$\begin{tikzcd}
 \cM_{g,n}^k \ar[hook]{r} \ar{d}{F}& \Mgb^k \ar{d}{\ov F} \\
 \cM_g^k \ar[hook]{r} &\ov\cM_g^k
\end{tikzcd}$$ 
We get an induced commutative diagram of homomorphisms
\begin{equation}\label{E:com-maps}
\begin{tikzcd}
\frac{\Pic(\cM_{g,n}^k)}{\Pic^{\ta}(\cM_{g,n}^k)} & \frac{\Pic(\Mgb^k)}{\Pic^{\ta}(\Mgb^k)}   \ar[swap]{l}{[\res_{g,n}]} \\
\frac{\Pic(\cM_g^k)}{\Pic^{\ta}(\cM_g^k)}   \ar{u}{[F^*]} & \frac{\Pic(\ov\cM_g^k)}{\Pic^{\ta}(\ov\cM_g^k)} \ar[swap]{u}{[\ov F^*]}\ar{l}{[\res_{g}]}
\end{tikzcd}
\end{equation}
where we have used that the pull-back homomorphism $\ov F^*:\Pic(\ov\cM_g^k)\to \Pic(\Mgb^k)$  sends tautological line bundles into tautological line bundles (see e.g. \cite[\S 3]{AC87}).
By parts  \eqref{T:Pic-taut2} and \eqref{T:Pic-taut3} of Theorem \ref{T:Pic-taut}, we know that the homomorphisms $[F^*]$,  $[\res_{g,n}]$ and $[\res_g]$ appearing in \eqref{E:com-maps} are isomorphisms. 
By the commutativity of the diagram \eqref{E:com-maps}, we conclude that  $[\ov F^*]$ is an isomorphism, and we are done.

\un{Step III:} \: \eqref{E:splitPic} holds true for any $n>0$ and  \eqref{E:tor-pull} holds true.

As observed above, we can (and will) restrict to $g\geq 2$ also in the proof of \eqref{E:splitPic}.
Since $\Pic^{\ta}(\Mgb^k)$ is torsion-free by part \eqref{T:Pic-taut1}, the natural homomorphism 
\begin{equation}\label{E:alphagn}
\alpha_{g,n}: \Pic(\Mgb^k)_{\tor}\hookrightarrow  \Pic(\Mgb^k) \twoheadrightarrow \frac{\Pic(\Mgb^k)}{\Pic^{\ta}(\Mgb^k)}
\end{equation}
is injective. Note that the decomposition \eqref{E:splitPic} holds true if and only if $\alpha_{g,n}$ is also surjective (hence an isomorphism). Step I implies that $\alpha_g:=\alpha_{g,0}$ is an isomorphism. 

Consider now the following commutative diagram
$$\begin{tikzcd}
\Pic(\ov\cM_g)_{\tor} \ar{r}{\ov F^*} \ar{d}{\alpha_g} & \Pic(\Mgb^k)_{\tor} \ar{d}{\alpha_{g,n}}\\
\frac{\Pic(\ov\cM_g^k)}{\Pic^{\ta}(\ov\cM_g^k)} \ar{r}{[\ov F^*]} & \frac{\Pic(\Mgb^k)}{\Pic^{\ta}(\Mgb^k)}
\end{tikzcd}$$
where we have used that $\ov F^*$  sends torsion line bundles into torsion line bundles.  Since $\alpha_g$ is an isomorphism by Step I, $[\ov F^*]$ is an isomorphism by Step II and $\alpha_{g,n}$ is injective, we conclude from the above diagram that $\alpha_{g,n}$ is an isomorphism (so that \eqref{E:splitPic} holds true) and that  \eqref{E:tor-pull} holds true.

\end{proof}

\section{Proof of the Main Results}\label{S:conclu}

The aim of this section is to prove Theorem A and Theorem B of the introduction.

Let us first collect in the following proposition all the properties of the two commutative group algebraic spaces $\Pi_S^{\tau}:\PIC^{\tau}_{\Mgb^S/S}\to S$ and $\Pi_S: \PIC_{\Mgb^S/S}\to S$.  

\begin{prop}\label{P:propPic}
Let $S$ be a scheme. 
\begin{enumerate}
\item \label{P:propPic1} $\Pi_S:\PIC_{\Mgb^S/S}\to S$ is a commutative group scheme,  locally quasi-finite over $S$ and it satisfies the valuative criterion of properness (hence it is separated).
\item \label{P:propPic2} $\Pi_S^{\tau}: \PIC^{\tau}_{\Mgb^S/S}\to S$ is an open and closed sub-group scheme of $\PIC_{\Mgb^S/S}$, which is finite over $S$.
\end{enumerate}
\end{prop}
Note however that $\Pi_S:\PIC_{\Mgb^S/S}\to S$ is not universally closed (hence not proper) since it is not quasi-compact. 
\begin{proof}
Since $\Mgb^S \to S$ is proper and smooth with geometrically connected fibers, it follows from Theorem \ref{T:Pictau} that $\Pi_S:\PIC_{\Mgb^S/S}\to S$ is a commutative group algebraic space,  locally of finite type and equidimensional over $S$, which satisfies the valuative criterion for properness (hence it is separated).  Moreover, Proposition \ref{P:tau-tors} gives that the relative dimension of $\Pi_S$  is zero, which implies that 
$\Pi_S:\PIC_{\Mgb^S/S}\to S$ is locally quasi-finite. We deduce that $\PIC_{\Mgb^S/S}$ is a scheme by \cite[\href{https://stacks.math.columbia.edu/tag/03XX}{Tag 03XX}]{stacks-project}.

Again  from Theorem \ref{T:Pictau}, it follows that $\PIC^{\tau}_{\Mgb^S/S}$ is an open and closed sub-group algebraic space of $\PIC_{\Mgb^S/S}$, which is proper over $S$. From the properties of $\PIC_{\Mgb^S/S}$ already shown, we deduce that $\PIC^{\tau}_{\Mgb^S/S}$ is a  scheme locally quasi-finite over $S$, which then implies that $\Pi_S^{\tau}$ is a finite morphism (being proper).
\end{proof}

The main ingredient in the proof of Theorem A is the following flatness result, which follows from the main result  of the Appendix.

\begin{teo}\label{T:flat}
Let $k$ be an algebraically closed field of characteristic $p>0$ and let $W(k)$ be the ring of Witt vectors on $k$. 
Consider the the commutative group scheme  
$$\Pi^{\tau}_{W(k)}:\PIC^{\tau}_{\Mgb^{W(k)}/W(k)}\to \Spec W(k).$$
Then we have that
\begin{enumerate}[(i)]
\item \label{T:flat1} $\Pi^{\tau}_{W(k)}$ is flat along the zero section;
\item \label{T:flat2} $\Pi^{\tau}_{W(k)}$ is flat if $p>2$.
\end{enumerate}
\end{teo}
\begin{proof}
It follows from Theorem \ref{T:main-Ray}, using that the absolute ramification index of $W(k)$ is $1$ and that the $\Pi^{\tau}_{W(k)}$ is equidimensional by Theorem \ref{T:Pictau}\eqref{T:Pictau3} and the fact that $f_{W(k)}:\Mgb^{W(k)}\to \Spec W(k)$ is smooth. 
\end{proof}

As an application of the above Theorem, we can prove the following

\begin{prop}\label{P:Pictau}
Let $k$ be any field.  Then we have that
\begin{enumerate}[(i)]
\item \label{P:Pictau1}  $\PIC^{\tau}_{\Mgb^k/k}$  is \'etale over $k$.
\item \label{P:Pictau2} If $\char(k)\neq 2$ then  $\PIC^{\tau}_{\Mgb^k/k}$  is the trivial $k$-group scheme.
\end{enumerate}
\end{prop}
\begin{proof}
Since the formation of  $\PIC^{\tau}_{\Mgb^k/k}$  commutes with field extensions, we can assume that $k$ is algebraically closed. 
We will distinguish two cases according to the characteristic of  $k$.

\un{I Case:} $\char(k)=0$.

In this case, $\PIC^{\tau}_{\Mgb^k/k}$  is the trivial $k$-group scheme by Proposition \ref{P:tau-tors}\eqref{P:tau-tors3}.

\un{II Case:} $\char(k)=p>0$.

Consider the commutative group scheme  
$$\Pi^{\tau}_{W(k)}:\PIC^{\tau}_{\Mgb^{W(k)}/W(k)}\to \Spec W(k),$$
where $W(k)$ is the ring of Witt vectors on $k=\ov k$. 

 The generic fiber of $\Pi_{W(k)}^{\tau}$ is equal to 
 $$\Pi_{F(W(k))}^{\tau}:\PIC^{\tau}_{\Mgb^{F(W(k))}/F(W(k))}\to \Spec F(W(k)),$$
  where $F(W(k))$ is the fraction field of $W(k)$. Since $\char F(W(k))=0$, it follows from the previous case  that $\PIC^{\tau}_{\Mgb^{F(W(k))}/F(W(k))}$ 
  is the trivial group scheme, or in other words that $\PIC^{\tau}_{\Mgb^{W(k)}/W(k)}$ is generically trivial over $\Spec W(k)$. This implies that the flat closure of $\PIC^{\tau}_{\Mgb^{F(W(k))}/F(W(k))}$ inside
  $\PIC^{\tau}_{\Mgb^{W(k)}/W(k)}$ (see \cite[Prop. 2.8.5]{EGAIV2}) is the image of the zero section of $\Pi_{F(W(k))}^{\tau}$. 
  
  
  Now, if $p>2$ then Theorem \ref{T:flat} implies that  $\PIC^{\tau}_{\Mgb^{W(k)}/W(k)}$ is the flat closure of its generic fiber, and hence it coincides with the image of its zero section (i.e. it is the trivial group scheme over $\Spec W(k)$). By passing to the central fiber, we get that 
  $\PIC^{\tau}_{\Mgb^k/k}$  is the trivial $k$-group scheme. 
   
 If, instead $p=2$, then Theorem \ref{T:flat} implies that there exists an open neighborhood $U$ of the image of the zero section of  $\PIC^{\tau}_{\Mgb^{W(k)}/W(k)}$ which is the flat closure of its generic fiber, and hence $U$ must coincide with the image of its zero section. But this implies that 
 the image of the zero section is a connected component of  $\PIC^{\tau}_{\Mgb^{W(k)}/W(k)}$, being open, closed and connected. By passing to the central fiber, we get that $\PIC^{\tau}_{\Mgb^k/k}$  is a $k$-group scheme of finite type over $k$ (by Proposition \ref{P:propPic}\eqref{P:propPic2}) having  the image of the zero section as a connected component, which
implies that $\PIC^{\tau}_{\Mgb^k/k}$ is \'etale over $k$.


\end{proof}

\begin{cor}\label{C:Pico}
Let $k$ be any field.  
\begin{enumerate}[(i)]
\item \label{C:Pico1} $\PIC^o_{\Mgb^{k}/k}$  is is the trivial $k$-group scheme.
\item \label{C:Pico2} $H^1(\Mgb^k,\cO_{\Mgb^k})=0$.
\end{enumerate}
\end{cor}
\begin{proof}
Since $\PIC^{\tau}_{\Mgb^k/k}$  is \'etale over $k$ by Proposition \ref{P:Pictau}, we have that the connected component of  $\PIC^{\tau}_{\Mgb^k/k}$ which contains the identity element, i.e. $\PIC^{o}_{\Mgb^k/k}$, must be the trivial $k$-group scheme. 
Using that $H^1(\Mgb^k,\cO_{\Mgb^k})=T_e \PIC_{\Mgb^{k}/k}$ by Proposition \ref{P:smPic}\eqref{P:smPic1}, we conclude that $H^1(\Mgb^k,\cO_{\Mgb^k})=0$. 
\end{proof}

\begin{rmk}\label{R:vanH1}
The vanishing of $H^1(\Mgb^{\bbC},\cO_{\Mgb^{\bbC}})$  (which then implies the result over any field of characteristic zero) follows from the vanishing of $H^1(\Mgb^{\bbC},\bbC)$ shown by Arbarello-Cornalba \cite[Thm. 2.1]{AC98} together with the Hodge decomposition 
$$H^1(\Mgb^{\bbC},\bbC)=H^1(\Mgb^{\bbC},\cO_{\Mgb^{\bbC}})\oplus H^0(\Mgb^{\bbC}, \Omega^1_{\Mgb^{\bbC}}).$$
An alternative proof of the vanishing of $H^1(\Mgb^{\bbC},\cO_{\Mgb^{\bbC}})$ was given by Okawa-Sano  \cite[Thm. 3.2]{OS} using the Hodge symmetry 
$$\dim H^1(\Mgb^{\bbC},\cO_{\Mgb^{\bbC}})=\dim H^0(\Mgb^{\bbC}, \Omega^1_{\Mgb^{\bbC}}).$$
Both proofs ultimately relies on the  Harer's vanishing theorem \cite{Har86}. 
\end{rmk}

We can now give a proof of Theorem A and Theorem B of the introduction.

\begin{proof}[Proof of Theorem A]
Part \eqref{T:mainThm1}: observe that, since $\un{\Lambda_{g,n}}_S$ is flat over $S$, the morphism $\Phi^S_{g,n}$ is an isomorphism if and only if $\Phi^{\ov{k(s)}}$ is an isomorphism for every $s\in S$ (see  \cite[(17.9.5)]{EGAIV4}).

Since $\un{\Lambda_{g,n}}_{k(s)}$ is \'etale over $\ov{k(s)}$ and $\Lambda_{g,n}$ is a torsion-free group, the morphism   $\Phi^{\ov{k(s)}}$ is an isomorphism if and only if 
\begin{equation}\label{E:cond-fib1}
\PIC^{\tau}_{\Mgb^{\ov{k(s)}}/\ov{k(s)}} \: \text{ is the trivial $\ov{k(s)}$-group scheme, }
\end{equation}
and 
\begin{equation}\label{E:cond-fib2}
\Lambda_{g,n} \xrightarrow{\Phi_{\ov{k(s)}}(\ov{k(s)})} \Pic(\Mgb^{\ov{k(s)}})\twoheadrightarrow \NS^{\nu}(\Mgb^{\ov{k(s)}})  \: \text{ is an isomorphism.} 
\end{equation}
Now, if $\char(k(s))\neq 2$ then condition \eqref{E:cond-fib1} holds by Proposition \ref{P:Pictau}\eqref{P:Pictau2}. On the other hand,  if $g\leq 5$ then condition \eqref{E:cond-fib1} holds by Proposition \ref{P:Pictau}\eqref{P:Pictau1} together with 
$$\PIC^{\tau}_{\Mgb^k/k}(k)=\NS(\Mgb^k)_{\tor}=\Pic(\Mgb^k)_{\tor}=\{0\},$$ 
which follows by Proposition \ref{P:tau-tors}\eqref{P:tau-tors1} and Theorem \ref{T:Pic-taut}\eqref{T:Pic-taut4}. Finally, condition \eqref{E:cond-fib2} holds by Theorem \ref{T:Pic-taut}: part \eqref{T:Pic-taut1} implies that $\Phi_{\ov{k(s)}}(\ov{k(s)})$ maps $\Lambda_{g,n}$ isomorphically onto the 
tautological subgroup $\Pic^{\ta}(\Mgb^{\ov{k(s)}})\subseteq \Pic(\Mgb^{\ov{k(s)}})$ and part \eqref{T:Pic-taut4} implies that $\Pic^{\ta}(\Mgb^{\ov{k(s)}})$ maps isomorphically onto $\NS^{\nu}(\Mgb^{\ov{k(s)}})$. 

Let us now prove part \eqref{T:mainThm2}. Since $S$ is connected, we have that $\un{\Lambda_{g,n}}(S)=\Lambda_{g,n}$. By the definition of the morphisms $\ov P_{g,n}(S)$  and $\Phi_{g,n}^S$, we have the following commutative diagram
$$\begin{tikzcd}
0 \ar{r} & \Pic(S) \ar{r}{f_S^*}  & \Pic(\Mgb^S) \ar{r} & \PIC_{\Mgb^S/S}(S)\\
& & \Lambda_{g,n}\ar{u}{\ov P_{g.n}(S)} \ar[equal]{r} & \un{\Lambda_{g,n}}(S) \ar[swap]{u}{\Phi_{g.n}^S(S)}  
\end{tikzcd}$$
where the first row is exact (see \cite[Chap. 8, Prop. 4]{BLR}). We now conclude using that $\Phi_{g,n}^S(S)$ is an isomorphism by part \eqref{T:mainThm1}. 
\end{proof}


\begin{rmk}\label{R:bad}
Without the assumption of Theorem A (namely that $S$ is a $\bbZ[1/2]$-scheme or that $g\leq 5$), the proof of Theorem A gives that the tautological $S$-morphism 
$\Phi_{g,n}^S$ is an open and closed embedding. Moreover, $\Pi_S:\PIC_{\Mgb^S/S}\to S$ is unramified  (by Corollary \ref{C:Pico}) and we have an isomorphism of $S$-groups schemes
$$\PIC^\tau_{\Mgb^S/S}\times \un{\Lambda_{g,n}}_S\xrightarrow{\cong}\PIC_{\Mgb^S/S}.$$
However, we do not know how to prove in this case that the unramified group scheme $\Pi_S^{\tau}:\PIC^{\tau}_{\Mgb^S/S}\to S$ is trivial. 
\end{rmk}

\begin{proof}[Proof of Theorem B]
This follows from Theorem \ref{T:Pic-taut} using that $\Pic(\Mgb^k)_{\tor}=0$ if $\char(k)\neq 2$ by Theorem A. 
\end{proof}

\section{Some consequences}\label{S:cons}

The aim of this section is to draw some consequences from the main results of the paper. 

We first observe that  the complex analytic first Chern class of $\Mgb^{\bbC}$ with values  on the singular second cohomology group of the complex analytic stack $\Mgb^{\an}$  is an isomorphism (see  \cite{Beh04} for a discussion of singular cohomology of DM topological stacks).

\begin{prop}\label{P:c1an}
The complex analytic first Chern class 
$$c_1^{\an}:\Pic(\Mgb^{\bbC}) \to H^2(\Mgb^{\an},\bbZ)$$
is an isomorphism. 
\end{prop}
\begin{proof}
By the GAGA theorem for the Picard group (see \cite[Sec. 5.2]{Toe}), we have that 
$$\Pic(\Mgb^{\bbC})\cong \Pic(\Mgb^{\an})=H^1(\Mgb^{\an}, \cO^*_{\Mgb^{\an}}).$$
We now conclude using the long exact sequence associated to the exponential sequence
$$
0\to \bbZ\xrightarrow{\cdot 2\pi i} \cO_{\Mgb^{\an}}\xrightarrow{\exp} \cO^*_{\Mgb^{\an}}\to 0,
$$
together with the vanishing
$$
H^1(\Mgb^{\an},\cO_{\Mgb^{\an}})=H^2(\Mgb^{\an},\cO_{\Mgb^{\an}})=0,
$$
which follow from  \cite[Thm. 2.1, Thm. 2.2]{AC98} and the Hodge decomposition (see also \cite[Thm. 3.2]{OS}). 

\end{proof}

We now investigate the \'etale first Chern class of $\Mgb^{k}$ with values  on the \'etale second cohomology group of  $\Mgb^{k}$ (see \cite{Ols07} for a discussion of \'etale cohomology of algebraic stacks).

\begin{prop}\label{P:c1et}
Let $k$ be an algebraically closed field with $\char(k)\neq 2$ and let $l$ be a prime different from $\char(k)\neq 2$. The \'etale first Chern class 
$$c_1^{et}:\Pic(\Mgb^{k})_{\bbZ_l} \to H_{et}^2(\Mgb^{k},\bbZ_l)$$
is an isomorphism. 
\end{prop}
\begin{proof}
We will distinguish two cases.

\un{I Case:} $\char(k)=0$.

Since both the domain and codomain of $c_1^{et}$ are compatible with base change of algebraically closed fields, it is enough to treat the case $k=\bbC$. 
In this case, the result follows from the fact that $c_1^{\an}$ is an isomorphism by Proposition  \ref{P:c1an}, together with the Artin's comparison theorem \footnote{the case of smooth complex schemes is treated in \cite[III, Thm. 3.12]{Mi}; the case of a smooth complex stack $\cX$ follows from the case of schemes using that the \'etale cohomology $H^{\bullet}_{et}(\cX,-)$ (resp. the singular cohomology $H^{\bullet}(\cX^{\an},-)$)  of a constant sheaf can be computed using a simplicial scheme $X_{\bullet}$ (resp. a simplicial $X_{\bullet}^{\an}$ topological space) obtained from an atlas $X\to \cX$ (see \cite[Thm 4.7]{Ols07} for \'etale cohomology and \cite{Beh04} for singular cohomology).}
\begin{equation*}
H^2(\Mgb^{\an},\bbZ)\otimes \bbZ_l \cong H^2_{et}(\Mgb^{\bbC}, \bbZ_l),
\end{equation*}
which is compatible with the cycle map. 

\un{II Case:} $\char(k)=p>0$.

%
We may suppose that $k$ is the residue field of a strictly Henselian ring $R$ with field of fractions of characteristic zero. Let $K$ be an algebraic closure of the field of fractions of $R$. Consider the smooth and proper morphism $\Mgb^R\to \Spec R$ and the cartesian diagram
\begin{equation}\label{E:cospec}
\begin{tikzcd}
\Mgb^{K}\ar{r}{g}\ar{d}&\Mgb^R\ar{d}{\pi}&\ar{l}[swap]{\iota}\ar{d}\Mgb^k\\
\Spec K\ar{r}&\Spec R&\ar{l}\Spec k.
\end{tikzcd}\end{equation} We want to show that the restriction maps
$$
H^i (\Mgb^R,\mathbb Z/m\mathbb Z)\xrightarrow{g^*} H^i(\Mgb^K,\mathbb Z/m\mathbb Z),\quad H^i (\Mgb^R,\mathbb Z/m\mathbb Z)\xrightarrow{\iota^*} H^i(\Mgb^{k},\mathbb Z/m\mathbb Z)
$$
are isomorphisms for any $m$ coprime with $p$. It is a well-known fact for schemes, we now explain how to extend the result to stacks. First, we claim that $g_*\mathbb Z/m\mathbb Z=\bbZ/m\mathbb Z$ and $R^ig_*\mathbb Z/m\mathbb Z=0$ for $i>0$. It can be checked on a smooth atlas of $U\to \Mgb^R$. Since a smooth morphism of schemes is locally acyclic, the claim follows by \cite[Chap V, Lemme 1.5]{SGA412} (here we are using $m$ coprime with $p$). In particular, the Leray spectral sequence attached to $g$ and $\mathbb Z/m\mathbb Z$ degenerates. Hence, the restriction map $g^*$ is an isomorphism. On the other hand, the cohomology group $H^i (\Mgb^R,\mathbb Z/m\mathbb Z)$ is canonically isomorphic to the stalk of $R^i\pi_*(\mathbb Z/m\mathbb Z)$ at the special point. By proper base change theorem for stacks (see \cite[Theorem 1.3]{Ols05}) applied to the second square in \eqref{E:cospec}, the restriction map $\iota^*$ is an isomorphism. In particular, we obtain the isomorphisms 
$$
H^2 (\Mgb^R,\mathbb Z_l)\xrightarrow{g^*} H^2(\Mgb^K,\mathbb Z_l),\quad H^2 (\Mgb^R,\mathbb Z_l)\xrightarrow{\iota^*} H^2(\Mgb^{k},\mathbb Z_l).
$$
By Theorem A, the analogous restrictions for Picard groups are also isomorphisms. Then, Case II follows from Case I together with the fact that the restriction maps commute with the first Chern classes.
\end{proof}

The final topic we address in this section is the computation of the divisor class group $\Cl(\ov M_{g,n}^k)$ of the coarse moduli space $\phi:\Mgb^k\to \ov M_{g,n}^k$,  which is an integral normal projective $k$-variety with finite quotient singularities  (see \cite{Mum77, Gie82}). 

Observe that, if $(\ov M_{g,n}^k)^{*}$ is an open and smooth subset of $\ov M_{g,n}^k$ whose complement has codimension at least two, the restriction homomorphism gives an isomorphism
\begin{equation}\label{E:rest}
	\res:\Cl(\ov M_{g,n}^k)\xrightarrow{\cong} \Cl((\ov M_{g,n}^k)^{*})=\Pic((\ov M_{g,n}^k)^{*}). 
\end{equation}
Moreover, since $\phi_*(\cO_{\Mgb^k})=\cO_{\ov M_{g,n}^k}$, the pull-back along the morphism $\phi$ gives an injective morphism 
\begin{equation}\label{E:phi*}
	\phi^*:\Pic((\ov M_{g,n}^k)^{*})\hookrightarrow \Pic(\phi^{-1}(\Pic((\ov M_{g,n}^k)^{*})))=\Pic(\Mgb^k). 
\end{equation}
Hence, by combining  \eqref{E:rest} and \eqref{E:phi*}, we get an injective homomorphism 
\begin{equation}\label{E:phibar}
	\wt\phi^*:\Cl(\ov M_{g,n}^k)\xrightarrow[\res]{\cong} \Pic((\ov M_{g,n}^k)^{*}) \stackrel{\phi^*}{\hookrightarrow} \Pic(\Mgb^k).
\end{equation} 
The image of $\wt\phi^*$ is described in the following result. 

\begin{prop}\label{P:Cl}
	Let $k$ be an algebraically closed field with $\char(k)\neq 2$ and let $g\geq 1$. The divisor class group $\Cl(\ov M_{g,n}^k)$ is identified, via the injective homomorphism $\wt\phi^*$ of \eqref{E:phibar}, with the subgroup of $\Pic(\Mgb^k)$ generated by 
	\begin{itemize}
		\item $\{2\lambda, \lambda+\delta_{1,\emptyset}, \{\delta_{a,A}: (a,A)\neq (1,\emptyset), (g-1,[n])\} \}$ if $g+n\geq 4$ or $(g,n)=(3,0)$; 
		\item $\{2\lambda, 2\psi_1, \lambda+\delta_{1,\emptyset}, \delta_{\irr}=10\lambda-2\delta_{1,\emptyset}\}$  if $(g,n)=(2,1)$;
		\item $\{2\lambda, \delta_{\irr}=12\lambda, 2\psi_1=2\psi_2, 2\delta_{1,\emptyset}=2\psi_1-2\lambda\}$ if $(g,n)=(1,2)$;
		\item $\{2\lambda, 2\delta_{1}, \delta_{\irr}=10\lambda-2\delta_{1}\}$ if $(g,n)=(2,0)$;
		\item $\{\delta_{\irr}=12\lambda=12\psi_1\}$ if $(g,n)=(1,1)$. 
	\end{itemize}
	In particular, we have that 
	$$
	\frac{\Pic(\Mgb^k)}{\wt\phi^*(\Cl(\ov M_{g,n}^k))}\cong 
	\begin{sis}
		\bbZ/2\bbZ & \: \text{if }g+n\geq 4 \: \text{or }(g,n)=(3,0); \\
		\bbZ/2\bbZ \times \bbZ/2\bbZ & \: \text{if }(g,n)=(2,1), (1,2), (2,0);  \\
		\bbZ/12\bbZ & \text{ if }(g,n)=(1,1). 
	\end{sis} 
	$$
\end{prop}
The above result was proved by Arbarello-Cornalba in \cite[Prop. 2]{AC87} for $g\geq 3$ and $k=\bbC$. The case $g=0$ is excluded by the above Proposition for trivial reasons: $\phi:\ov\cM_{0,n}^k\to \ov M_{0,n}^k$ is an isomorphism
(see \cite[Chap. XII, Prop. (2.5)]{GAC2}) and hence 
$\Cl(\ov M_{0,n}^k)=\Pic(\ov M^k_{0,n})\xrightarrow{\phi^*} \Pic(\ov \cM^k_{0,n})$ is an  isomorphism. 
\begin{proof}
	We will simply write $\Mgb$ for $\Mgb^k$ and $\ov M_{g,n}$ for $\ov M_{g,n}^k$. 
	Denote by $\ov M_{g,n}^0$ the open subset of $\ov M_{g,n}$ consisting of stable $n$-pointed curves with trivial automorphism group. Note that $\phi$ induces an isomorphism between $\phi^{-1}(\ov M_{g,n}^0)$ and $\ov M_{g,n}^0$, so that $\ov M_{g,n}^0\subseteq (\ov M_{g,n})_{\sm}$. 
	We now divide the proof in several cases.
	
	\un{I Case:} $g+n\geq 4$.
	
	In this case, let $(\ov M_{g,n})^*$ be the intersection between the smooth locus $(\ov M_{g,n})_{\sm}$ of $\ov M_{g,n}$ and the open subset $\ov M_{g,n}^0\cup D_{1,\emptyset}^0$ of $\ov M_{g,n}$, where $D_{1,\emptyset}$ is the boundary divisor of $\ov M_{g,n}$ consisting of curves with an (unpointed) elliptic tail and $D_{1,\emptyset}^0$ is the open subset of $D_{1,\emptyset}$ consisting of stable $n$-pointed curves of $D_{1,\emptyset}$
	whose unique non-trivial automorphism is the involution  $\iota$ which is the elliptic involution on the elliptic tail and the identity on the rest of the curve. The singular locus of $\ov M_{g,n}$ has codimension at least two, because the variety $\ov M_{g,n}$ is normal, and the complement of $\ov M_{g,n}^0\cup D_{1,\emptyset}^0$ has codimension at least two, by Proposition \ref{P:aut-nontriv}. Hence, the open subset $(\ov M_{g,n})^*$ is smooth and its complement has codimension at least two. In particular, we get an injective homomorphism as in \eqref{E:phibar}.
	Since $\char(k)\neq 2$, the coarse moduli space $\phi:\phi^{-1}((\ov M_{g,n})^*)\to (\ov M_{g,n})^*$ is a good moduli space in the sense of Alper \cite{Al}. By \cite[Thm. 10.3]{Al}, a line bundle $L$ on $\Mgb$ descends to a line bundle on $(\ov M_{g,n})^*$ if and only if, for any stable $n$-pointed curve $C\in D_{1,\emptyset}^0$, we have that the involution $\iota$ of $C$ acts as the identity on the fiber $L_{|C}$ of $L$ at $C$. As proved in \cite[p. 51]{HM82}, $\iota$ acts as $-1$ on $\lambda_{|C}$ and on $(\delta_{1,\emptyset})_{|C}$, and as $+1$ on the fibers above $C$ of all the other tautological line bundles of $\Pic(\Mgb)$ (see Theorem B). 
	This yields the conclusion in this case.

	\un{II Case:} $g+n=3$, i.e. $(g,n)=(3,0), (2,1), (1,2)$. 
	
	In this case we define 
	$$(\ov M_{g,n})^*=(\ov M_{g,n}^0\cup D_{1,\emptyset}^0 \cup \ov{H}_{g,n}^0)\cap (\ov M_{g,n})_{\sm},$$
	where $\ov{H}_{g,n}^0$ is the locus of stable $n$-pointed hyperelliptic curves $(C,p_1,\ldots, p_n)$ whose marked points (for $n=1,2$) are fixed by the hyperelliptic involution $\iota$ and whose automorphism group is generated by $\iota$. 
	 The singular locus of $\ov M_{g,n}$ has codimension at least two, because the variety $\ov M_{g,n}$ is normal, and the complement of $\ov M_{g,n}^0\cup D_{1,\emptyset}^0\cup \ov{H}_{g,n}^0$ has codimension at least two, by Proposition \ref{P:aut-nontriv}. Hence, the open subset $(\ov M_{g,n})^*$ is smooth and its complement has codimension at least two. In particular, we get an injective homomorphism as in \eqref{E:phibar}.
	Therefore, arguing as in the previous case, a line bundle $L$ on $\Mgb$ descends to a line bundle on $(\ov M_{g,n})^*$ if and only if  for any stable $n$-pointed curve $C\in D_{1,\emptyset}^0\cup \ov{H}_{g,n}^0$, we have that 
	the involution $\iota$ of $C$ acts as the identity on the fiber $L_{|C}$ of $L$ at $C$. If $C\in D_{1,\emptyset}^0$ this gives exactly the conditions of the I Case. On the other hand, if  $C\in H_{g,n}^0$ then $\iota$ acts as $(-1)^g$ on $\lambda_{|C}=\Lambda^g H^0(C,\omega_C)$ (since $\iota$ acts as $-\id$ on $H^0(C,\omega_C)$) and as $-1$ on 
	$(\psi_i)_{|C}=(\omega_C)_{|p_i}$ for $1\leq i \leq n$ (since if a local equation for $C$ at $p_i$ is given by $y^2=f(x)$ with $p_i=(0,\alpha_i)$ then $(\omega_C)_{|p_i}$ is generated by ${\rm d} y$ and $\iota$ sends $y$ into $-y$), 
	and it acts as  $+1$  on the fibers above $C$ of all the other tautological line bundles of $\Pic(\Mgb)$ (see Theorem B). 
	This yields the conclusion in this case.

	\un{III Case:} $(g,n)=(1,1)$. 
	
	Since $\ov M_{1,1}=\bbP^1$, we have that $\Cl(\ov M_{1,1})=\Pic(\ov M_{1,1})\cong \bbZ$, generated by the point  $D_{\irr}$ corresponding to the nodal elliptic curve. Hence, we conclude using that (see Theorem A) 
	$$\Pic(\ov\cM_{1,1})=\frac{\bbZ[\lambda, \psi_1,\delta_{\irr}]}{(12\lambda-\delta_{\irr}, \lambda-\psi_1)},$$
	together with the fact that $\phi^*(\cO(D_{\irr}))=\delta_{\irr}$.
	
	\un{IV Case:} $(g,n)=(2,0)$. 
	
	Consider the commutative diagram with exact rows
	\begin{equation}\label{E:diag-g2}
		\begin{tikzcd}
			0 \ar{r} & \bbZ[\delta_{\irr}, \delta_1] \ar{r}  & \Pic(\ov\cM_2)=\frac{\bbZ[\lambda, \delta_{\irr}, \delta_1]}{(10\lambda=\delta_{\irr}+2\delta_1)} \ar{r} & \Pic(\cM_2)=\frac{\bbZ[\lambda]}{(10\lambda=0)}\ar{r} & 0\\
			0 \ar{r} & \bbZ[D_{\irr}, D_1] \ar{r} \ar[hookrightarrow]{u}{\wt\phi^*} & \Cl(\ov M_2)\ar{r}  \ar[hookrightarrow]{u}{\wt\phi^*} & \Cl(M_2)\ar{r} \ar[hookrightarrow]{u}{\wt\phi^*} & 0\\
		\end{tikzcd}
	\end{equation}
	where we have used Theorem A in the explicit presentation of $\Pic(\ov \cM_2)$, and  $\{D_{\irr}:=\phi(\Delta_{\irr}), D_1:=\phi(\Delta_1)\}$ are the boundary divisors of $\ov M_2$.
	Since $\Cl(M_2)=\bbZ/5\bbZ$ (see \cite[Cor. 3.8]{GV}), it follows that $\Cl(M_2)$ is identified,  via the injective homomorphism $\wt\phi^*$, with the subgroup of $\Pic(\cM_2)$ generated by $2\lambda$. 
	We conclude using  the commutative diagram \eqref{E:diag-g2} and the fact that $\wt\phi^*(D_{\irr})=\delta_{\irr}$ and $\wt\phi^*(D_1)=2\delta_1$.
\end{proof}

\begin{rmk}\label{R:gerbe}
	The stacks $\ov\cM_{1,1}^k$ and $\ov \cM_2^k$ are $\mu_2$-gerbes over their rigidifications $\rho_{1,1}:\ov\cM_{1,1}^k\to \ov \cM_{1,1}^k\fatslash \mu_2$ and $\rho_{2}:\ov\cM_{2}^k\to \ov \cM_{2}^k\fatslash \mu_2$
	that remove the generic involution of every curve in $\ov \cM_{1,1}^k$ and $\ov \cM_2^k$.
	Using similar arguments as in the proof of the above Proposition, it can be checked that 
	$$
	\Pic(\ov \cM_{1,1}^k\fatslash \mu_2)=\frac{\bbZ[2\lambda, 2\psi_1,\delta_{\irr}]}{(\delta_{\irr}=6(2\lambda), 2\lambda=2\psi_1)}\cong \bbZ[2\lambda]  \stackrel{\rho_{1,1}^*}{\hookrightarrow}\bbZ[\lambda]\cong 
	\frac{\bbZ[\lambda, \psi_1,\delta_{\irr}]}{(\delta_{\irr}=12\lambda, \lambda=\psi_1)}=\Pic(\ov\cM_{1,1}^k), 
	$$
	$$
	\Pic(\ov\cM_2^k\fatslash \mu_2)=\frac{\bbZ[\lambda, \delta_{\irr}, \delta_1]}{(10\lambda=\delta_{\irr}+2\delta_1)} \xrightarrow[\cong]{\rho_2^*} \Pic(\ov\cM_2^k).
	$$
	In particular, the gerbe $\rho_{1,1}$ is trivial while the gerbe $\rho_2$ is non-trivial (compare with  \cite[Prop. 4.7]{GV2}).  
\end{rmk}

\begin{prop}\label{P:aut-nontriv}Assume $g\geq 1$. Let $(\Mgb^k)^0$ be the open substack on $\Mgb^k$ consisting of stable $n$-pointed curves with trivial automorphism group and let $\ov{\mt X}_{g,n}^k$ be the union of divisors in the complement $\Mgb^k\setminus (\Mgb^k)^0$. Then, we have that
	$$
	\ov{\mt X}_{g,n}^k=\begin{cases}
		\Delta_{1,\emptyset},&\text{if }g+n\geq 4,\\
		\Delta_{1,\emptyset}\cup \ov{\mathcal H}^k_{g,n},&\text{if }g+n=3,
	\end{cases}
	$$
	where $\Delta_{1,\emptyset}$ is the boundary divisor of $\Mgb^k$ consisting of curves with an unpointed elliptic tail and $\ov{\mathcal H}_{g,n}^k$ is the closed substack of stable $n$-pointed hyperelliptic curves of genus $g$.
	
	Furthermore, the following facts hold true:
	\begin{enumerate}[(i)]
		\item the sublocus $\Delta_{1,\emptyset}^0\subset\Delta_{1,\emptyset}$, of those curves whose only non-trivial automorphism is the elliptic involution on the elliptic tail and the identity on the rest of the curve, is an open and dense substack;
		\item the sublocus $(\ov{\mathcal H}_{g,n}^k)^0\subset\ov{\mathcal H}_{g,n}^k$, of those curves whose only non-trivial automorphism is the hyperelliptic involution, is an open and dense substack.
	\end{enumerate} 
\end{prop}

\begin{proof}We remove the reference to the field $k$ from the notation. We first prove the following
	
	\textbf{Claim.} Let $\Mg^0$ be the open subset in $\Mg$ of smooth curves without non-trivial automorphisms and let $\mt Y_{g,n}$ be the union of irreducible components in $\Mg\setminus\Mg^0$ of codimension at most $1$. Then, we have the following equality
	\begin{equation}\label{E:cMg0}
		\mt Y_{g,n}=\begin{cases}
			\emptyset,&\text{if }g+n\geq 4,\\
			\mathcal H_{g,n}&\text{if }g+n\leq 3,
		\end{cases}
	\end{equation}
	where $\mt H_{g,n}$ is the moduli stack of smooth $n$-pointed hyperelliptic curves. Furthermore, the stack $\mt H_{g,n}$ coincides with $\Mg$, if $(g,n)=(1,1), (2,0)$, and it is a divisor, if $(g,n)=(3,0),(2,1)$. It contains an open and dense subset $\mt H_{g,n}^0$ consisting of $n$-pointed hyperelliptic curves whose automorphism group is generated by the hyperelliptic involution.
	
	We divide the proof of the claim in several cases depending on the genus. But, first, we need some preparation.
	Consider, the natural morphism $\pi:\Mg\to \mt M_{g,n-1}$, which forgets the last marking. It is a smooth morphism with geometrically integral fibers. Since $\pi$ is smooth, the substack $\mt Y_{g,n}$ is contained in the inverse image  $\pi^{-1}(\mt Y_{g,n-1})$. 
	Denote by $\mu:\mt Y_{g,n}\to\mt Y_{g,n-1}$ the restriction of the morphism $\pi$ to $\cY_{g,n}$. It is a quasi-finite morphism. Indeed, an automorphism $\varphi\in\Aut(C,p_1,\ldots,p_{n-1})$ in $\mt M_{g,n-1}$ lifts to an automorphism in $\Aut(C,p_1,\ldots,p_{n-1},p_n)$ if and only if $\varphi$ fixes $p_n$. Since the automorphism group of $C$ is a finite group and the fixed points of any automorphism of $C$ are a finite set, the morphism $\mu$ must be quasi-finite.
	
	In particular, by comparing their dimensions, we get the following criterion
	\begin{equation}\label{E:crit-dim}
		\text{if } \cY_{g,n-1} \text{ is either empty or a divisor in } \mathcal M_{g,n-1}, \text{then }		\cY_{g,m} \text{ is empty, for any }m\geq n.
	\end{equation}

	\underline{I Case:} $g=1$ (and $n\geq 1$).
	
	Assume $n=1$. In this case, any elliptic curve has the elliptic involution, i.e. $\mathcal H_{1,1}=\mathcal M_{1,1}$ and $\mathcal M_{1,1}^0=\emptyset$. Moreover, the locus $\mathcal H_{1,1}^0$ is a non-empty (and, so, dense) open substack of $\mathcal M_{1,1}$, because an elliptic curve with $j$-invariant different from either $0$ or $1728$ has the automorphism group generated by the elliptic involution (see \cite[Theorem III.10.1]{Si}). Assume $n=2$. Since the generic elliptic curve $C$ has only the elliptic involution as non-trivial automorphism, any lift of $C$ over $\mu:\cY_{1,2}\to\cY_{1,1}=\mathcal M_{1,1}$ must be an hyperelliptic $2$-pointed curve, i.e. $C\in \mt H^0_{1,2}$. Since $\mu$ is quasi-finite, we must have that $\cY_{1,2}$ is a divisor and it coincides with $\mt H_{1,2}$. The cases $n\geq 3$ follow from Criterion \eqref{E:crit-dim}.

	\underline{II Case:} $g=2$.
	
	The argument follows the same lines of the previous case. Assume $n=0$. In this case, any curve is hyperelliptic, i.e. $\mathcal H_{2}=\mathcal M_{2}$ and $\mathcal M_{2}^0=\emptyset$. The statement about $\mathcal H_{2}^0$ follows from \cite{Poonen}. Assume $n=1$. Since the generic curve $C$ in $\mt M_2$ admits only the hyperelliptic involution as non-trivial automorphism, any lift of $C$ over $\mu:\cY_{2,1}\to\cY_2$ must be an hyperelliptic $1$-pointed curve, i.e. $C\in \mt H^0_{2,1}$.  Since $\mu$ is quasi-finite, we must have that $\cY_{2,1}$ is a divisor and it coincides with $\mt H_{2,1}$. The cases $n\geq 2$ follow from Criterion \eqref{E:crit-dim}.
	
	\underline{III Case:} $g=3$.
	
	Assume $n=0$. The equality \eqref{E:cMg0} has been proved by Popp \cite[Theorem (2) at page 106]{Popp}. The statement about $\mt H_{3}^0$ follows from \emph{loc.cit.}, where it has been also proved that $\mt H_{3}$ is a divisor in $\mt M_3$. The cases $n\geq 1$ follow from Criterion \eqref{E:crit-dim}. 
	
	\underline{IV Case:} $g\geq 4$. 
	
	Assume $n=0$. The equality \eqref{E:cMg0} has been proved by Popp \cite[Theorem (1) at page 106]{Popp}. The cases $n\geq 1$ follow from Criterion \eqref{E:crit-dim}.
	
	We are now ready to prove the assertion of the proposition. 
	
	Assume first $g+n\geq 4$. By Claim, we know that $\ov{\cX}_{g,n}$ is contained in the boundary $\Delta:=\Mgb\setminus\Mg$ of singular curves. A generic point $(C,p_1,\ldots,p_n)$ in $\Delta$ consists of either an irreducible curve with a unique node $N$ or a curve with a unique separating node $N$, whose normalization is the disjoint union of two curves $C_1$ and $C_2$, where $C_1$, resp. $C_2$, is an $m$-pointed, resp. $(n-m)$-pointed, smooth curve of genus $i$, resp. $g-i$ (where $m\geq 2$, resp. $n-m\geq 2$, if $i=0$, resp. $i=g$). In any of these cases, the automorphism group of $(C,p_1,\ldots,p_n)$ is equal to the automorphism group of the normalization of $C$, which fixes the inverse images of the markings and the two points over the the node $N$. In other words, they may be described as the stabilizers of objects either in $\mt M_{g-1,n+2}$ or $\mt M_{i,m+1}\times\mt M_{g-i,n-m+1}$. Using the equality \eqref{E:cMg0}, one may deduce that the generic stable $n$-pointed curve in $\Delta$ has trivial automorphism group, except when $C=C_1\cup C_2$, with $C_1$ unpointed smooth curve of genus $1$, i.e. it is an object in $\Delta_{1,\emptyset}$. This concludes the proof of the first statement for the case $g+n\geq 4$. Arguing as above, we see that the automorphism group of a generic curve in $\Delta_{1,0}$ is equal to the automorphism group of a generic object $(C_1,C_2)$ in $\mathcal M_{1,1}\times \mathcal M_{g-1,n+1}$. By Claim, the automorphism group of the generic curve $C_1$, resp. $C_2$, is generated by the elliptic involution, resp. is trivial. Hence, $\Delta_{1,\emptyset}^0$ is open and dense in $\Delta_{1,\emptyset}$. The proof for $g+n=3$ follows the same arguments.
\end{proof}

\begin{rmk}When $\char(k)=0$, the open subsets $(\ov M_{g,n}^k)^*$ defined in the proof of Proposition \ref{P:Cl} for $g+n\geq 3$ coincide with the smooth locus of $\ov M_{g,n}$, i.e.
	$$
	(\ov M_{g,n}^k)_{\sm}=\begin{cases}
		\ov M_{g,n}^0\cup D_{1,\emptyset}^0&\text{if }g+n\geq 4,\\
		\ov M_{g,n}^0  \cup D_{1,\emptyset}^0\cup \ov{H}_{g,n}^0&\text{if }g+n=3,
	\end{cases}
	$$
	see \cite{Cor89}, \cite{Cor87}, \cite[Chap. XII, Prop. (2.5)]{GAC2}. As far as we know, there is no proof (or disproof) of the analogous statements in positive characteristic. For some partial results in this direction, see \cite{Popp}.
\end{rmk}

\section{Appendix: $p$-torsion of the Picard group space for algebraic stacks} \label{S:App}
 
The goal of this Appendix is to show the following result, which is a generalization to algebraic stacks of a result of Raynaud \cite[Thm. 4.1.2]{Ray79}.

\begin{teo}\label{T:main-Ray}
Let $(R,\mathfrak{m})$ be a complete DVR, with field of fractions $K$ of characteristic $0$ and algebraically closed residue field $k:=R/\mathfrak m$ of characteristic $p>0$ and let $e\geq 1$ be the absolute ramification index (i.e. $pR=\mathfrak m^e$). 
Let $\mathcal X$ be an algebraic stack and $\pi:\mathcal X\to\Spec(R)$ be a proper and cohomologically flat in degree zero morphism. Consider the Picard group space\footnote{The assumptions on $\pi$ guarantee that the functor $\PIC_{\mathcal X/R}^{\tau}$ is representable by an algebraic space of finite type over $\Spec(R)$ and its formation commutes with base change, see Theorem \ref{T:Pictau}\eqref{T:Pictau1}.}
	$$
	\Pi^{\tau}_{\cX/R}:\PIC_{\mathcal X/R}^{\tau}\to \Spec(R).
	$$
If $\Pi^{\tau}_{\cX/R}$ is equidimensional\footnote{which is always true if the special fiber of $\pi$ (and hence also the geometric generic fiber) is  normal, see Theorem \ref{T:Pictau}\eqref{T:Pictau3}}, then we have that
\begin{enumerate}[(i)]
	\item if $e<p-1$, the morphism $\Pi^{\tau}_{\cX/R}$ is flat,
	\item if $e\leq p-1$, the morphism $\Pi^{\tau}_{\cX/R}$ is flat along the zero section.
\end{enumerate}
\end{teo}

Our proof will be a copy of the proof of \cite[Theorem 4.1.2]{Ray79}. Most of the work is already done in \emph{loc.cit.}, we just need to generalize few results.

\subsection{Cartier Theory}
In this subsection, we recall some results on the Cartier theory of formal groups, following  \cite[\S2]{Ray79} and \cite{Zink} (see also the english translation \cite{Zinkweb}).

Let $R$ be a commutative unitary $\bbZ_{(p)}$-algebra and  denote by $\text{Nil}(R)$ the category of nilpotent $R$-algebras, i.e. $R$-algebras whose elements are nilpotent.
We denote by \textbf{Ab} the (abelian) category of abelian groups.

\begin{defin}\label{D:form-gr-ring}A \emph{(abelian) formal group functor over $R$} is a functor $G:\text{Nil}(R)\to \textbf{Ab}$. We denote by $\mathcal G$ ($=\mathcal G_R$) the category of formal group functors over $R$.\\
A \emph{(abelian) formal group $G$ over $R$} is an exact formal group functor over $R$, which commutes with arbitrary direct sums. We denote by $\mathcal G^{o}$ ($=\mathcal G_R^{o}$) the category of formal groups over $R$.
\end{defin}

\begin{rmk}We remark that a formal group is always formally smooth, because it is right exact.
\end{rmk}

Any $R$-module $M$ can be seen as a nilpotent $R$-algebra $M_{Nil}:=M$, by setting $M_{Nil}^2=0$. In particular, the ring $R$, as free $R$-module, gives a nilpotent $R$-algebra $R_{Nil}$. For any formal group $G$, the abelian group $G(R_{Nil})$ carries a canonical $R$-module structure. We denote by $\lie(G)$ such $R$-module (for more details see \cite[\S4]{Zink}). 
If $R=K$ is a field, we define the dimension of $G$ as 
$
\dim_KG:=\operatorname{rank}_K\lie(G). 
$

\begin{rmk}Given a group scheme $H$ over $R$, we may define a formal group functor
	$$
	\begin{array}{cccl}
	\widehat{H}:&\text{Nil}(R)&\to& \textbf{Ab}\\
	&A&\mapsto&\ker\{H(R\oplus A)\to H(R)\}
	\end{array}
	$$
where $R\oplus A$ is the unitary commutative $R$-algebra, whose multiplication is defined by
\begin{equation}\label{E:aug-alg} 
(r_1,a_1)\cdot (r_2,a_2):=(r_1r_2,r_1a_2+r_2a_1+a_1a_2).
\end{equation}
\end{rmk}
Following the above remark, we would like to have a definition of smoothness for formal group compatible with the smoothness for group scheme.
\begin{defin}A formal group $G$ over $R$ is \emph{smooth} if $\lie(G)$ is a coherent $R$-module. We denote by $\mt G^{sm}$ ($=\mt G^{sm}_R$) the subcategory of $\mt G^o$ of formal smooth groups over $R$.
\end{defin}

The Cartier theory describes the formal groups over $R$ in term of modules over a certain ring that we now introduce. We denote by:
\begin{itemize}
	\item $W(R)$ the \emph{ring of Witt vectors} with coefficients in $R$.
	\item $\wpr(=\wpr_{R}):\text{Nil}(R)\to \textbf{Ab}$ the \emph{Witt formal group functor}, which sends a nilpotent $R$-algebra $A$ to $W(A)$.
	\item $\fwr(=\fwr_R)$ the \emph{Witt formal group}: it is the sub-functor of the Witt pseudo-ring such that
	$$
	\fwr(A):=\{(a_0,\ldots,a_n,\ldots)\in W(A)\mid a_i=0 \text{ for all except finitely many $i$}\}.
	$$
	It is a formal group over $R$.
	\item $\mathbb E(=\mathbb E_R)$ the endomorphisms ring of the formal group  $\fwr$, which acts on the right on $\fwr$.
\end{itemize}
The ring $\mathbb E$ contains elements $V$, $F$ and $\{[a]\,|\,a\in R\}$ satisfying certain relations (see \cite[Prop. 2.1.1(ii)]{Ray79}). Furthermore, they generate the entire ring. More precisely, any element in the ring $\mathbb E$ has a unique expression (see \cite[Prop. 2.1.1(i)]{Ray79})
$$
\sum_{m,n\geq 0}V^m[a_{m,n}]F^n
$$
such that for every $m$ it holds  that $a_{m,n}=0$ for all except finitely many $n$. Given an $\mathbb E$-module $M$, the quotient $M/VM$ has a canonical structure as $R$-module, given by $a.(m+VM):=[a]m+VM$ for any $a\in R$ and $m\in M$. We denote such $R$-module by $\lie(M)$. 
\begin{defin}Let $M$ be an $\mathbb E$-module.
\begin{enumerate}[(a)]
\item $M$ is \emph{without $V$-torsion} if $V:M\to M$ is injective.
\item $M$ is \emph{$V$-flat} if $M$ is without $V$-torsion and $\lie(M)$ is a flat $R$-module.
\item $M$ is \emph{$V$-coherent} if $M$ is without $V$-torsion and $\lie(M)$ is a coherent $R$-module.
\end{enumerate}
\end{defin}
%
%
%
%
%
Given a formal group functor $G$, the group of homomorphisms of formal group functors
$$
\text{Hom}_{\mathcal G}\Big(\fwr,G\Big)
$$
has a natural structure of left $\mathbb E$-module induced by the action of $\mathbb E$ on $\fwr$. On the other hand, given an $\mathbb{E}$-module $M$, we may define a formal group functor
$$
\begin{array}{cccl}
\fwr\otimes_{\mathbb E} M:&\text{Nil}(R)&\to &\textbf{Ab}\\
&A&\mapsto& \fwr(A)\otimes_{\mathbb E} M
\end{array}
$$
Let $\mathcal M(=\mathcal M_R)$ be the abelian category of left $\mathbb E$-modules. By what discussed above, there exist two functors
\begin{equation}
\begin{array}{ccl}
\mathcal G&\xrightarrow{\alpha}&\mathcal M\\
G&\mapsto&\text{Hom}_{\mathcal G}\Big(\fwr,G\Big)
\end{array},\quad 
\begin{array}{ccl}
\mathcal M&\xrightarrow{\beta}&\mathcal G\\
M&\mapsto&\fwr\otimes_{\mathbb E} M
\end{array}
\end{equation}

We present the main theorem of Cartier theory. First of all, we fix some notations.
\begin{defin}Let $\mathcal M^{o}$ be the subcategory of $\mathcal M$ of $\mathbb E$-modules $M$ such that
	\begin{enumerate}[(a)]
		\item $M$ is $V$-flat,
		\item $M$ is complete with respect to the $V$-adic filtration, i.e.
		$M=\lim_{\leftarrow}M/V^nM.
		$
	\end{enumerate}
	Let $\mathcal M^{sm}$ be the subcategory of $\mathcal M^{o}$ whose elements are $V$-coherent.
	\end{defin}

\begin{teo}[Cartier]\label{T:cartier} The functor $\alpha$ induces equivalence of categories
	$$
	\alpha:\mathcal G^{o}\xrightarrow{\cong}\mathcal M^{o}\quad \text{ and }\quad \alpha:\mathcal G^{sm}\xrightarrow{\cong}\mathcal M^{sm},
	$$
whose inverses are given by $\beta$. Furthermore, the functors $\alpha$ and $\beta$ are compatible with the formation of $\lie(-)$, i.e. 
$$
\begin{array}{ccccll}
\lie(\beta(M))&:=&\lie(\fwr\otimes_{\mathbb E}M)&=&\lie(M)&\text{for any }M\in \mathcal M^o\\
\lie(\alpha(G))&:=&\lie\left(\operatorname{Hom}_{\mathcal G}\Big(\fwr,G\Big)\right)&=&\lie(G)&\text{for any }G\in \mathcal G^o
\end{array}
$$
as $R$-modules.
\end{teo}

\begin{proof}The conditions of being without $V$-torsion and complete correspond to the definition of $V$-reduced $\mathbb E$-module, see \cite[Definition 4.20]{Zink}. Then the theorem is equivalent to \cite[Theorem 4.23]{Zink}.
\end{proof}

\subsection{Proof of Main Theorem}Unless otherwise stated, in this subsection $R$ will be a commutative unitary $\mathbb{Z}_{(p)}$ algebra. Let $\mathcal X$ be a quasi-compact algebraic stack over $R$ and denote by $\text{Nil}(\mathcal X)$ be the category of nilpotent $\mathcal O_{\mathcal X}$-algebras. As for the affine schemes, we make the following
\begin{defin} An \emph{(abelian) formal group $G$ over $\cX$} is an exact functor
$$
G:\operatorname{Nil}(\cX)\to \textbf{Ab},
$$
which commutes with arbitrary direct sums.
\end{defin}
Note that for any morphism $f:U=\Spec(\widetilde{R})\to\cX$, the pull-back 
$$
G|_{U}:\operatorname{Nil}(\widetilde{R})\cong \operatorname{Nil}(\Spec(\widetilde{R}))\to \operatorname{Nil}(\cX) \to\textbf{Ab}
$$
is a formal group over $\widetilde{R}$ in the sense of Definition \ref{D:form-gr-ring}. In particular, the assignment $(U\to\cX)\mapsto \lie(G|_U)$ defines a sheaf of $\oo_{\cX}$-modules, 
which will be  denoted by $\lie(G)$. 

\begin{defin}A formal group $G$ over $\cX$ is \emph{smooth} if $\lie(G)$ is a coherent $\oo_{\cX}$-module.
\end{defin}
Note that if $\cX=\Spec(\widetilde R)$ is an affine scheme, the above definitions agree with the definitions of (smooth) formal group over $\widetilde R$ given in the previous subsection. 

For any $\mathcal A\in\text{Nil}(\mathcal X)$, we define a sheaf $G[\mathcal A]$ of abelian groups over $\mathcal X$ (with respect to the lisse-\'etale topology) as it follows:
$$
(U\to\cX)\longmapsto(G|_U)(\mathcal A|_U).
$$
We want to study its cohomology groups $H^i(\mathcal X,G[A])$, with respect to the lisse-\'etale topology. We will show that they can be computed using a suitable \v{C}ech complex by combining Raynaud and an argument due to Brochard. Consider a (smooth) atlas $p:U^0\to\mathcal X$ with $U_0$ affine. The fiber product $U^0\times_{\mathcal X}U^0$ is an algebraic space. Fix an atlas $W^1\to U^0\times_{\mathcal X}U^0,$ with $W^1$ affine scheme and set $U^1:=U^0\bigsqcup W^1$. We obtain a truncated hypercover
$$\xymatrix{
	U^1 \ar@<-.15cm>[r] \ar@<+.15cm>[r] & \ar[l]U^0\ar[r]& \mathcal X
}$$
Arguing as in \cite[proof Lemma A.2.1]{Br2}, let $\mathcal U^\bullet$ be $1$-coskeleton of the above truncated hypercover. It is an hypercover for $\mathcal X$ in the sense of Verdier, see \cite[\href{https://stacks.math.columbia.edu/tag/09VU}{Tag 09VU}]{stacks-project}. Furthermore, for any $i\geq 0$ the term $U^i$ of the hypercover $\mathcal U^\bullet$ is affine. If $U^i=\Spec(R_i)$, we denote by $\fwr_i(=\fwr_{R_i})$ the Witt formal group on $R_i$ and by $\mathbb E_i(=\mathbb E_{R_i})$ its endomorphisms ring. Consider the left $\mathbb E_i$-module
\begin{equation}\label{E:Mi}
M_G^i:=\alpha\left(G_{R_i}\right)=\text{Hom}\Big(\fwr_i,G_{R_i}\Big).
\end{equation}
Let $M_G^{\bullet}$ be \v{C}ech complex of $\mathbb{E}$-modules
\begin{equation}\label{E:Cech-MG}
\cdots\to M_G^{i}\to M_G^{i+1}\to\cdots
\end{equation}
associated to the hypercover $\mathcal U^\bullet$. Furthermore, we denote by $\fwr(\mathcal A)\otimes_\mathbb{E} M_G^\bullet$, the \v{C}ech complex of abelian groups
\begin{equation}\label{E:Cech-MGi}
\cdots\to\fwr_i(\mathcal A|_{U^i})\otimes_{\mathbb E_i}M_G^{i}\longrightarrow\fwr_{i+1}(\mathcal A|_{U^{i+1}})\otimes_{\mathbb E_{i+1}}M_G^{i+1}\to\cdots
\end{equation}
associated to the hypercover $\mathcal U^\bullet$. We now show, generalizing \cite[Thm. 2.7.2]{Ray79},  that it represents the cohomology of the abelian sheaf $G[\mathcal A]$.

\begin{prop}\label{Ray272} Let $\mathcal X$ be a algebraic stack of finite type over $R$ and let $G$ be a smooth formal group on $\cX$. 
On the derived category of abelian groups $D(\textbf{Ab})$ there exists a canonical isomorphism
	$$
	H^\bl(\mathcal X, G[\mathcal A])\cong\fwr(\mathcal A)\otimes_\mathbb{E}M_G^\bullet,
	$$
	functorial in $\mathcal A\in \operatorname{Nil}(\mathcal X)$.
\end{prop}

\begin{proof}By Cartier Theorem \ref{T:cartier}, 
	there exists a canonical isomorphism of abelian groups 
	$$
	\fwr_{i}(\mathcal A|_{U^i})\otimes_{\mathbb{E}_i}M_G^i\cong (G|_{ U^i})(\mathcal A|_{U^i}),
	$$
	functorial in $\mathcal A$. Furthermore, it gives an isomorphism of complexes (functorial in $\mathcal A$)
	\begin{equation}\label{E:iso-compl}
	\fwr(\mathcal A)\otimes_\mathbb{E}M_G^\bullet\cong C(\mathcal U^\bullet,G[\mathcal A]),
	\end{equation}
	where $C(\mathcal U^\bullet,G[\mathcal A])$ is the  \v{C}ech complex of the sheaf $G[\mathcal A]$ with respect to the hypercover $\mathcal U^\bullet$. Consider the spectral sequence:
	\begin{equation}\label{E:CartLer}
	\check{H}^p(H^q(\mathcal U^\bullet,G[\mathcal A]))\Rightarrow H^{p+q}(\mathcal X,G[\mathcal A]).
	\end{equation}
	with respect to the \v{C}ech cohomology of the abelian sheaf $G[\mathcal A]$ and the hypercover $\mathcal U^\bullet$ (see \cite[\href{https://stacks.math.columbia.edu/tag/09VY}{Tag 09VY}]{stacks-project}). We now claim that 
       \begin{equation}\label{E:van-aff}
       V \text{ affine }\Rightarrow H^i(V,G[\mathcal A])=0 \text{ for }i>0.
	\end{equation}
	This fact has been proved in \cite[Lemma 2.7.1]{Ray79} for cohomology with respect to the Zariski topology. However, the proof in \emph{loc.cit.} still works if we consider the lisse-\'etale cohomology and so the claim holds.

       The vanishing \eqref{E:van-aff} implies that the spectral sequence \eqref{E:CartLer} degenerates and it gives a canonical isomorphism (functorial in $\mathcal A$)
       \begin{equation}\label{E:deg-spect}
       H^p(C(\mathcal U^\bullet,G[\mathcal A]))=\check{H}^p(H^0(\mathcal U^\bullet,G[\mathcal A]))\cong H^p(\mathcal X, G[\mathcal A]).
       \end{equation}
       We conclude by combing \eqref{E:deg-spect} and \eqref{E:iso-compl}.
\end{proof}

The main difficulty in generalising the Raynaud results is that an algebraic stack $\cX$ can have infinite cohomological dimension. This fact implies that, even when $\pi:\cX\to \Spec(R)$ is proper, the cohomology of a given coherent sheaf cannot be, in general,  represented by a perfect complex over $\Spec(R)$, i.e. a bounded complex of finite projective $R$-modules. However, if we consider the truncated cohomology, there exists a weaker result, which will be enough for our scope.

The following result is a weaker (i.e. truncated) version of \cite[Thm. 2.7.3]{Ray79}, which however works also for algebraic stacks.

\begin{teo} \label{Ray273} Let $\mathcal X$ be an algebraic stack and $R$ be a complete  DVR over  $\bbZ_{(p)}$\footnote{Note that \cite[Thm. 2.7.3]{Ray79} works for any commutative $\bbZ_{(p)}$-algebra $R$ which is complete for the $p$-adic topology, while we need a much more restricted class of rings.}.
 Suppose we have:
	\begin{itemize}
		\item $\pi:\mathcal X\to \Spec(R)$ a flat and proper morphism,
		\item $G$ a smooth formal group    over $\cX$,
		\item an integer $n\geq 1$.
	\end{itemize}
	Then there exists a bounded complex $P^\bl$ of smooth formal  $R$-groups such that:
	\begin{enumerate}[(i)]
		\item\label{i:273-i} $\lie(P^\bl)$ is a perfect complex of $R$-modules.
		\item\label{i:273-ii} when restricting to the category $\text{Nil}(p\text{-}R)$ of nilpotent $R$-algebras killed by a power of $p$, there is an isomorphism of functors
		$$H^i(P^\bl) \cong R^i\pi_*(G):\text{Nil}(p\text{-} R)\to \textbf{Ab}$$
		for $0 \leq i\leq n-1$, compatible with any base change $R\to R'$, with $R'$ being any commutative (not necessarily Noetherian)  ring.
			\end{enumerate}
\end{teo}

\begin{proof}Consider the composition of functors
	$$
	R^{\bl}\pi_*(G):\text{Nil}(R)\xrightarrow{\pi^*}\text{Nil}(\mathcal X)\to D(\textbf{Ab})
	$$
	sending a nilpotent $\mathcal O_R$-algebra $\mathcal A$ to $H^\bl(\cX,G[\pi^*(\mathcal A)])$. We remark that it is not a sheaf in general. Following the notations of \eqref{E:Cech-MGi}, we have a canonical isomorphism
	$$
	\fwr(\mathcal A)\otimes_{\mathbb E} M_G^i\cong \fwr(\pi^*(\mathcal A)|_{U^i})\otimes_{\mathbb E_i}M_G^i,
	$$
	functorial in $\mathcal A\in \text{Nil}(R)$ and compatible with the pull-back maps $U^i\to U^j$ of affine schemes, see \cite[2.4.3]{Ray79}. This, together with Proposition \ref{Ray272}, implies that there is a canonical isomorphism in $D(\textbf{Ab})$, functorial in $\mathcal A\in\text{Nil}(R)$
	\begin{equation}\label{E:G=Cech}
	R^{\bl}\pi_*(G)(\mathcal A)\cong  \fwr(\mathcal A)\otimes_{\mathbb E} M_G^\bullet,
	\end{equation}
	where $ \fwr(\mathcal A)\otimes_{\mathbb E} M_G^\bullet$ is the \v{C}ech complex of abelian groups \eqref{E:Cech-MGi}. Consider the complex of $\mathbb E$-modules
	$$
	(C^\bullet_{n+1})^i:=\begin{cases}
	M_G^{i},& \text{if }i\leq n,\\
	\ker\{M_G^{n+1}\xrightarrow{d^{n+1}} M_G^{n+2}\}, &\text{if }i=n+1,\\
	0,&\text{otherwise.}
	\end{cases}
	$$
	We make the following claims.
	\begin{enumerate}
		\item\label{i:1} $C^\bullet_{n+1}$ is a complex of $V$-flat modules. In particular,
		$$
		H^i(\fwr\overset{\mathbb L}{\otimes}_\mathbb{E}C^\bullet_{n+1})=H^i(\fwr\otimes_\mathbb{E}C^\bullet_{n+1})=R^i\pi_*(G),
		$$
		for $0\leq i\leq n-1$ and it is compatible with arbitrary base changes $R\to R'$.
		\item\label{i:2} For $0\leq i\leq n+1$, then
		$$
		H^i(\lie(C^\bullet_{n+1}))=H^{i}(\lie(M_G^\bullet))=H^{i}(\mathcal X,\lie(G)).
		$$
		In particular, the cohomology groups of $\lie(C^\bullet_{n+1})$ are $R$-modules of finite type.
	\end{enumerate}
	Assume the claim holds. Then $\lie(C^\bullet_{n+1})$ is a bounded complex  of flat $R$-modules with finite type cohomology. Hence, by  Lemma \ref{P:br} below, 
	$\lie(C^\bullet_{n+1})$ is quasi-isomorphic to a perfect complex $L^\bullet$ of $R$-modules. 
	Then the theorem follows by applying \cite[Thm. 2.6.4, Ex. 2.6.5(ii)]{Ray79} at the (bounded from above) complex $C^\bullet_{n+1}$ of $\mathbb E$-modules and the perfect complex $L^\bl$ of $R$-modules.
	
	We now show the claim. The $\mathbb{E}$-modules $M_G^i$ are $V$-flats since $\pi$ is a flat morphism (see \cite[\S 2.5]{Ray79}) and so Claim\eqref{i:1} holds true for $ (C^\bullet_{n+1})^i$ for $i\neq n+1$. We have to show that the $\mathbb E$-module  $(C^\bullet_{n+1})^{n+1}= \ker(d^{n+1})$ is $V$-flat, i.e. that the $R$-module $\lie(\ker(d^{n+1}))$ is flat. Consider the factorization 
	$$
	d^{n+1}: M_G^{n+1}\stackrel{\ov{d^{n+1}}}{\twoheadrightarrow}\Im(d^{n+1})\hookrightarrow M_G^{n+2},
	$$
	and observe that $\ker(d^{n+1})=\ker(\ov{d^{n+1}})$. Consider the commutative diagram of $\mathbb E$-modules
	\begin{equation}\label{E:diag-E}	
	\begin{tikzcd}
		\ker(d^{n+1}) \arrow[hook]{r}\arrow[hook]{d}{V} & M_G^{n+1}\arrow[twoheadrightarrow]{r}{\ov{d^{n+1}}} \arrow[hook]{d}{V} & \Im(d^{n+1}) \arrow[hook]{d}{V} \\
		\ker(d^{n+1}) \arrow[hook]{r}\arrow[twoheadrightarrow]{d}& M_G^{n+1}\arrow[twoheadrightarrow]{r}{\ov{d^{n+1}}} \arrow[twoheadrightarrow]{d} & \Im(d^{n+1}) \arrow[twoheadrightarrow]{d} \\
		\lie(\ker(d^{n+1})) \arrow{r}& \lie(M_G^{n+1})\arrow{r}{\ov{d^{n+1}}} & \lie(\Im(d^{n+1}))
	\end{tikzcd}
	\end{equation}
	where the upper two rows are clearly exact and the three columns are exact by the definition of $\lie(-)$ and the fact that $M_G^{n+1}$ is without $V$-torsion (which implies that also $\ker(d^{n+1})\subseteq M_G^{n+1}$ is without 
	$V$-torsion) and that $\Im(d^{n+1})$ is without $V$-torsion being a submodule of $M_G^{n+2}$ which is without $V$-torsion. These facts and the snake lemma implies that the last row of \eqref{E:diag-E} is also exact, 
	and in particular we get an    
	inclusion 
	\begin{equation}\label{E:inc-ker}
	\lie(\ker(d^{n+1})) \hookrightarrow \lie(M_G^{n+1}).
	\end{equation}
	Now, from this inclusion, we deduce that the $R$-module $\lie(\ker(d^{n+1}))$ is flat since over a DVR flat is equivalent to torsion-free (\cite[\href{https://stacks.math.columbia.edu/tag/0539}{Tag 0539}]{stacks-project}), $\lie(M_G^{n+1})$ is flat and the property of being torsion-free is preserved by taking submodules. This concludes the proof of the Claim\eqref{i:1}.
	
	 On the other hand, Claim\eqref{i:2} is obvious for $i\leq n-1$. It remains to show the cases $i=n, n+1$. Take the exact sequence of $R$-modules
	\begin{equation}\label{E:ex}
	\xymatrix{
		\lie(\ker(d^{n+1}))\ar[r] &\lie(M_G^{n+1})\ar[rr]^{\lie(d^{n+1})}&& \lie(M_G^{n+2})\ar[r]& \lie(\text{coker}(d)^{n+1})\arrow[r]\ar@{=}[d]& 0\\
		&&&&\text{coker}(\lie(d^{n+1}))
	}
	\end{equation}
	The vertical equality follows because the functor $\lie(-):\mathcal M_R\to\operatorname{Mod}(R)$ of abelian categories is right-exact. By previous point, the above sequence is left-exact, i.e.
	\begin{equation}\label{E:ker-lie}
			\lie(\ker(d^{n+1}))=\ker(\lie(d^{n+1})).
	\end{equation}
	In particular,
	$$
	\begin{array}{ccccc}
	\text{Im}\Big\{\lie(M^n_G)\to\lie(\ker (d^{n+1}))\Big\}&=&\text{Im}\Big\{\lie(M^n_G)\to \lie(M^{n+1}_G)\Big\}&\overset{\text{def}}{=}&\text{Im}(\lie(d^n)),\\
	\ker\Big\{\lie(M^n_G)\to\lie(\ker (d^{n+1}))\Big\}&=&\ker\Big\{\lie(M^n_G)\to \lie(M^{n+1}_G)\Big\}&\overset{\text{def}}{=}&\ker(\lie(d^{n})).
	\end{array}
	$$
	From the first equality, we get
	$$
	H^{n+1}(\lie(C^\bullet_{n+1}))=\displaystyle\frac{\lie(\ker(d^{n+1}))}{\text{Im}(\lie(d^n))}=\displaystyle\frac{\ker(\lie(d^{n+1}))}{\text{Im}(\lie(d^n))}\overset{\text{def}}{=}
	H^{n+1}(\lie(M_G^\bullet))).
	$$
	Similarly, from the second equality, we obtain $H^{n}(\lie(C^\bullet_{n+1}))=H^{n}(\lie(M^\bullet_G)))$ and we are done.
\end{proof}
\begin{lem}\label{P:br}	Let $R$ be a Noetherian ring. Let $K^\bullet$ be a bounded complex of flat $R$-modules with finite type cohomology, then there exists a perfect complex $L^\bullet$ of $R$-modules and a quasi-isomorphism $L^\bullet\to K^\bullet$ such that
	$$
	H^i(R'\otimes_R L^\bullet)\cong 	H^i(R'\otimes_R K^\bullet)
	$$
for any $i$ and for any base change $R\to R'$, with $R'$ being any commutative (not necessarily Noetherian) ring.
\end{lem}

\begin{proof}See \cite[\S 5 Lemma 1 and 2]{Mum70}.	
\end{proof}

We now focus our attention on the smooth formal group    $\widehat{\mathbb G}_{m}$, i.e. the formal completion of the multiplicative group $\mathbb G_{m}$ along the identity section. We denote by $\widehat{\mathbb G}_{m,\mathcal X}$ the formal group functor induced on $\mathcal X$. 

Let $\pi:\mathcal X\to\Spec(R)$ be a proper and cohomologically flat morphism. Then, the push-forwards $R^i\pi_*\widehat{\mathbb G}_{m,\mathcal X}$ in low degrees along the morphism $\pi:\cX\to \Spec R$ can be described as it follows:
\begin{equation}\label{E:def-compl}
N\to \ker\left\{H^i(\mathcal X_{R\oplus N},\mathbb G_{m,\mathcal X_{R\oplus N}})\to H^i(\mathcal X_{R}, \mathbb G_{m,\mathcal X_{R}})\right\},\quad \text{ for }i=0,1,
\end{equation}
where $R\oplus N$ is the $R$-algebra defined as in \eqref{E:aug-alg}. We remark that, despite the notation, the functor $R^i\pi_*\widehat{\mathbb G}_{m,\mathcal X}$ is not a sheaf in general. We are now ready for

\begin{proof}[Proof of Theorem \ref{T:main-Ray}]From \eqref{E:def-compl}, we get immediately the following equalities
\begin{enumerate}[(a)]
\item $R^0\pi_*(\widehat{\mathbb G}_{m,\mathcal X})=\widehat{\mathbb G}_{m, R}$,
\item $R^1\pi_*(\widehat{\mathbb G}_{m,\mathcal X})$ is the completion $\widehat{\PIC}$ of $\PIC^\tau:=\PIC^\tau_{\mathcal X/R}$ along the identity section,
\item $\lie(\widehat{\mathbb G}_{m,\mathcal X})=\mathcal O_{\mathcal X}$.
\end{enumerate}
Since we are supposing that $\Pi^\tau_{R}$ is equidimensional, we have that 
\begin{equation}\label{E:qu-dim}
\dim_K R^1\pi_*(\widehat{\mathbb G}_{m,\mathcal X})|_K=\dim_K \widehat{\PIC}_K= \dim_k\widehat{\PIC}_k=\dim_k R^1\pi_*(\widehat{\mathbb G}_{m,\mathcal X})|_k.
\end{equation}
Let $P^\bl$ be a  bounded complex of smooth formal  $R$-groups, obtained by applying   Theorem \ref{Ray273}  to the smooth formal group $\widehat{\mathbb G}_{m,\mathcal X}$ and with  $n=3$. By \eqref{E:qu-dim},  we have that
\begin{equation}\label{E:qu-dimP}
\dim_K H^1(P^\bl_K)=\dim_kH^1(P^\bl_k).
\end{equation}
Assume $e\leq p-1$. By \cite[Corollary 3.3.3]{Ray79} and using \eqref{E:qu-dimP}, we get  that $R^1\pi_*(\widehat{\mathbb G}_{m,\mathcal X})=\widehat{\PIC}$ is flat over $R$ and the sheafification $\widetilde{R^2\pi_*}(\widehat{\mathbb G}_{m,\mathcal X})$ with respect to the fppf topology, is representable by a formal group $B$ over $R$. The first fact already proves that $\Pi^{\tau}_{\cX/R}:\PIC^\tau\to \Spec R$ is flat along the zero section if $e\leq p-1$ (in particular $(ii)$ holds). 

Let $\overline{\PIC^\tau_K}$ be the flat closure of $\PIC^\tau_K$ in $\Pi^{\tau}_{\cX/R}:\PIC^\tau\to \Spec R$. If we show that the group scheme $T=\PIC^\tau/\overline{\PIC^\tau_K}$ is trivial for $e<p-1$, we are done.
Fix an element $\overline{x}\in T(k)$. Since $k$ is algebraically closed, there exists a line bundle $\overline L\in \mathcal X_k$ representing the class $\overline{x}$. For proving $T=\{1\}$, it is enough to show that, up to a fppf base change, $\ov L$ lifts to a line bundle on $\mathcal X$. By Formal GAGA for stacks (see \cite{cnrgaga}), we can replace $\mathcal X$ with its completion $\widehat{\mathcal X}$ along the special fiber $\mathcal X_k$. The obstruction of lifting $\ov L$ over $\widehat{\cX}$ is an element in
$$
H^2(\widehat{\mathcal X}, \widehat{\mathbb G}_{m,\widehat{\mathcal X}})=R^2\pi_*(\widehat{\mathbb G}_{m,\widehat{\mathcal X}})(\mathfrak m).
$$
By construction, the obstruction defines an element $\eta\in B(\mathfrak m)$. Then arguing as in \cite[Proof of Theorem 4.1.2]{Ray79}, we get the assertion.
\end{proof}

\vspace{0,3cm}

\noindent {\bf Acknowledgments.} 
We are  grateful to Marco Antei, Marco Boggi, Andrea Di Lorenzo,  Gabriele Mondello, Nicola Pagani, Roberto Pirisi and Fabio Tonini for useful conversations. 

The first author acknowledges PRIN2017 CUP  E84I19000500006, and the MIUR Excellence Department Project awarded to the Department of Mathematics, University of Rome Tor Vergata, CUP E83C18000100006. 

The second author is a member of the CMUC (Centro de Matem\'atica da Universidade de Coimbra), where part of this work was carried over. 
The second author is member of the project EXPL/MAT-PUR/1162/2021 funded by FCT (Portugal), and of the projects PRIN-2017 ``Advances in Moduli Theory and Birational Classification"
and PRIN-2020 ``Curves, Ricci flat Varieties and their Interactions" funded by MIUR (Italy).

\bibliographystyle{alpha}
\bibliography{BiblioPicMgn}
\end{document}